\documentclass[twocolumn]{autart}

\usepackage{mathrsfs}
\usepackage{graphics} 
\usepackage{amsmath} 
\usepackage{amssymb}  
\usepackage{cite}
\usepackage{bm}
\usepackage{acronym}
\usepackage{paralist}
\usepackage{float}
\usepackage{color}
\usepackage{epstopdf}
\usepackage{tikz}
\usepackage{graphicx}
\usepackage{soul}

\usepackage{mathtools}

\newtheorem{Theorem}{Theorem}

\newtheorem{Lemma}{Lemma}

\newtheorem{Problem}{Problem}

\newtheorem{Remark}{Remark}

\newtheorem{Assumption}{Assumption}

\newtheorem{Definition}{Definition}

\DeclareMathOperator{\sign}{sign}

\def\IR{{\mathbb R}}

\begin{document}

\begin{frontmatter}

\title{Fixed-time Control under Spatiotemporal and Input Constraints: A Quadratic Program Based Approach\thanksref{footnoteinfo}} 

\thanks[footnoteinfo]{The authors would like to acknowledge the support of the Air Force Office of Scientific Research under award number FA9550-17-1-0284 and of the National Science Foundation under the award number 1942907. This paper was not presented at any IFAC meeting. \\
Ehsan Arabi was a Postdoctoral Research Fellow at the Department of Aerospace Engineering, University of Michigan, when this work was conducted. He is currently a Research Engineer at Ford Research and Advanced Engineering, Ford Motor Company, Dearborn, MI 48121, USA.}

\author[P1]{Kunal Garg}\ead{kgarg@umich.edu},    
\author[P2]{Ehsan Arabi}\ead{earabi@ford.com},    
\author[P1]{Dimitra Panagou}\ead{dpanagou@umich.edu}

\address[P1]{Department of Aerospace Engineering, University of Michigan, Ann Arbor, MI, 48109, USA}

\address[P2]{Research and Advanced Engineering, Ford Motor Company, Dearborn, MI 48121, USA}




\begin{abstract}
In this paper, we present a control synthesis framework for a general class of nonlinear, control-affine systems under spatiotemporal and input constraints. First, we study the problem of fixed-time convergence in the presence of input constraints. The relation between the domain of attraction for fixed-time stability with respect to input constraints and the required time of convergence is established. {It is shown that increasing the control authority or the required time of convergence can expand the domain of attraction for fixed-time stability.} Then, we consider the problem of finding a control input that confines the closed-loop system trajectories in a safe set and steers them to a goal set within a fixed time. To this end, we present a Quadratic Program (QP) formulation to compute the corresponding control input. We use slack variables to guarantee feasibility of the proposed QP under input constraints.
Furthermore, when strict complementary slackness holds, we show that the solution of the QP is a continuous function of the system states, and establish uniqueness of closed-loop solutions to guarantee forward invariance using Nagumo's theorem. We present two case studies, an example of adaptive cruise control problem and an instance of a two-robot motion planning problem, to corroborate our proposed methods.
\end{abstract}


\end{frontmatter}

\section{Introduction}
Driving the state of a dynamical system to a given desired point or a desired set in the presence of constraints is a problem of major practical importance. 
Constraints requiring the system trajectories to evolve in some \textit{safe} set at all times while visiting some goal set(s) are common in safety-critical applications. Constraints pertaining to convergence to a goal set within a fixed time often appear in time-critical applications, e.g., when a task must be completed within a given time interval. \textit{Spatiotemporal} specifications impose spatial (state) as well as temporal (time) constraints on the system trajectories. Figure \ref{fig:motivating example} shows a scenario of a motivating problem requiring a quadrotor to remain in a domain (acting as a safe set), which consists of the region bounded by the green boundary excluding the regions marked in red. Furthermore, the blue regions denote the goal sets, which the quadrotor is required to visit in a given time sequence. 

\begin{figure}[b]
    \centering
    \includegraphics[width=1\columnwidth,clip]{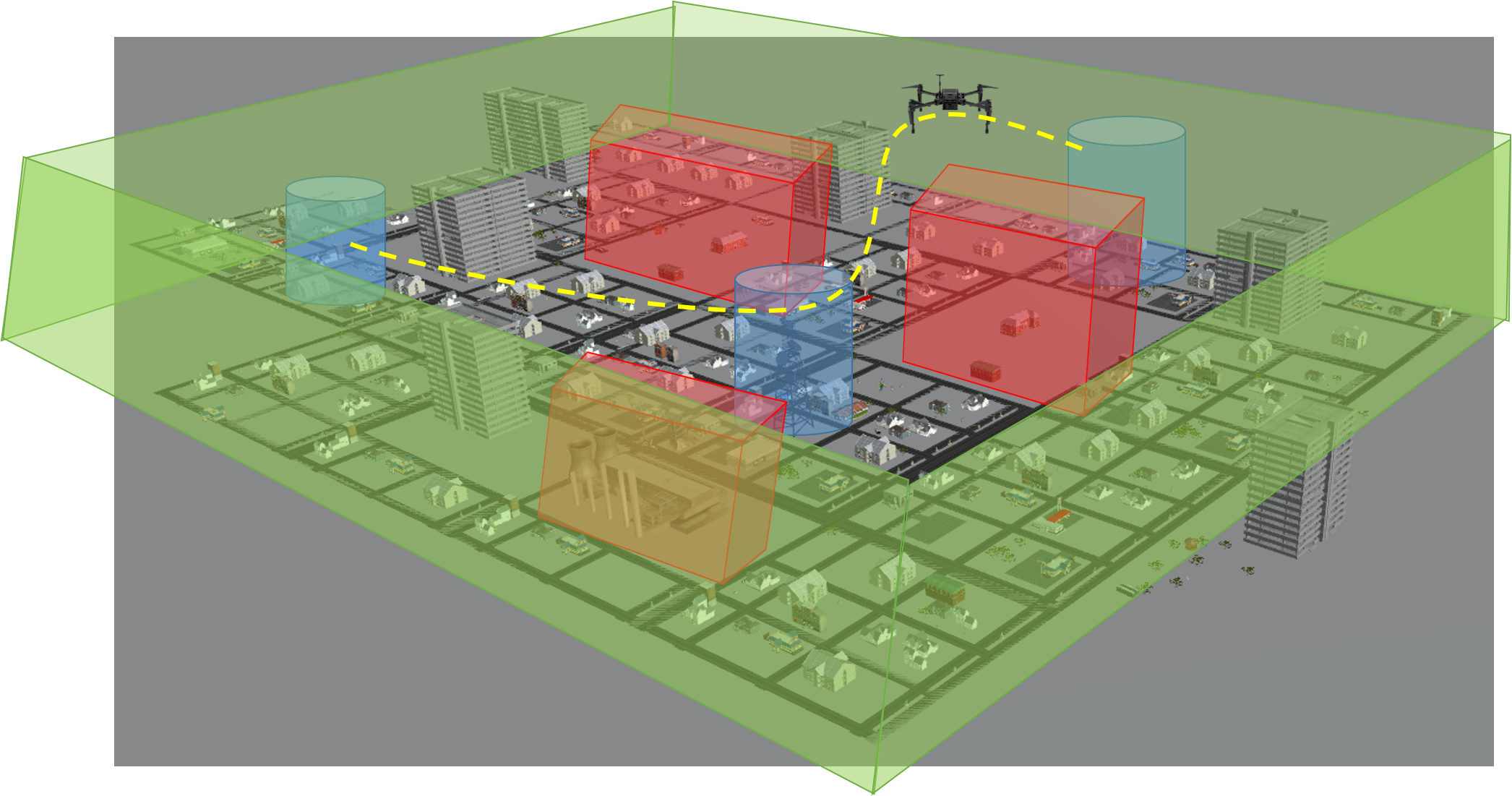}
    \caption{Motivating example scenario for a motion planning problem governed by spatiotemporal constraints.}\label{fig:motivating example}
\end{figure}

Safety in dynamical systems is typically realized as establishing that a desired set of safe states, or safe set, is \textit{forward invariant} under the system dynamics. The control objective reduces to designing a control law such that the closed-loop system trajectories remain always in the safe set. The approach in \cite{tee2009barrier} utilizes Lyapunov-like barrier functions to guarantee that the system output always remains inside a given set. More recently, in \cite{ames2017control}, conditions using Zeroing Control Barrier Functions (ZCBF) are presented to ensure forward invariance of a desired set. Various approaches have been developed to achieve convergence of system trajectories to desired sets or points while satisfying control input constraints. Methods such as Model Predictive Control (MPC) \cite{saska2014motion,grancharova2015uavs} as well as Control Lyapunov Functions (CLF) \cite{srinivasan2018control,li2018formally} have been studied extensively in the literature. Quadratic Program (QP)-based approaches have gained popularity for control synthesis, see for instance \cite{li2018formally,srinivasan2018control,ames2017control,rauscher2016constrained, cortez2019control}. These methods are suitable for real-time implementation as QPs can be solved efficiently \cite{cortez2019control,glotfelter2017nonsmooth,glotfelter2018boolean}. The authors in \cite{nguyen2016exponential} use an exponential barrier function in the QP formulation to guarantee safety of the closed-loop trajectories. In \cite{wang2017safety}, safety barrier certificates are presented to ensure scalable collision-free behavior in multi-robot systems. 

Concurrent forward invariance of a safe set and convergence to a goal set can be achieved via a combination of CLFs and Control Barrier Functions (CBFs), see e.g., \cite{ames2017control,romdlony2016stabilization}. However, the concurrent satisfaction of the corresponding conditions and the underlying control synthesis problem become more challenging in the presence of input constraints, such as actuator saturation, since the latter may affect the region of safety and fixed-time convergence of the system trajectories. Most of the aforementioned contributions address control design that achieves safety along with convergence to a desired goal set (or point), but without explicitly considering control input constraints. Such constraints are considered in \cite{ames2017control}, where performance and safety objectives are represented using CLFs and CBFs, respectively, along with control input constraints in a QP. 

Furthermore, most of the aforementioned work, with exception of \cite{srinivasan2018control,li2018formally}, deals with asymptotic or exponential convergence of the system trajectories to the desired goal point or goal sets. In contrast, Finite-Time Stability (FTS) is a concept that guarantees convergence in finite time. The authors in \cite{bhat2000finite} introduce necessary and sufficient conditions in terms of a Lyapunov function for the equilibrium of a continuous-time, autonomous system to exhibit FTS. Fixed-Time Stability (FxTS) \cite{polyakov2012nonlinear} is a stronger notion than FTS, where the time of convergence does not depend upon the initial conditions. For specifications involving temporal constraints and time-critical systems, the theory of finite- or fixed-time stability can be leveraged in the control design to guarantee that such specifications are met.
It has also been shown that a faster rate of convergence generally implies that the closed-loop system has better disturbance rejection properties \cite{bhat2000finite}, which further motivates the study of finite- and fixed-time stable systems. 
The authors of \cite{srinivasan2018control} formulate a QP for finite-time convergence to a desired set, however without considering input constraints. This limitation is removed in \cite{li2018formally}, where the authors consider a QP formulation incorporating input, safety and convergence constraints. The authors in \cite{lindemann2019control} use CBFs in a QP formulation to encode Signal-Temporal Logic (STL) specifications that impose reaching to a goal set within a finite time.
 

In this paper, we study the problem of reaching a given goal set $S_G$ {within a fixed time $T_{ud}$}, while remaining in a given safe set $S_S$ at all times, for a general class of nonlinear control-affine systems with input constraints. In the preliminary conference version \cite{garg2019control}, a QP formulation is proposed to compute the control input for fixed-time convergence under input and safety constraints, yet without any guarantees on the feasibility of the proposed method. Per its definition, {FxTS} of an equilibrium point from arbitrary initial conditions presumes unbounded control authority. To address the problem of {FxTS} in the presence of input constraints, new Lyapunov conditions from \cite{garg2021FxTSDomain} are utilized. When used in a QP, the new Lyapunov conditions introduce a slack term, which results in feasibility guarantees in the presence of input constraints. The contributions of the paper as follows:
\begin{itemize}
    \item First, FxTS conditions are utilized in a QP, and using Karush-Kuhn-Tucker (KKT) conditions, the closed-form expression for the optimal value of the slack term is computed for the case when the control input is saturated. Then, the relation between the Domain of Attraction (DoA) for FxTS, the input bounds, and the fixed time of convergence is established for a 1-D control-affine system. 
    \item Then, a novel QP formulation that utilizes Fixed-Time (FxT) CLFs and CBFs is proposed to synthesize controllers for nonlinear, control-affine systems, so that forward invariance of a safe set and FxT convergence of the system trajectories to a goal set is guaranteed. We use slack terms both in the safety and the FxT convergence constraints to ensure that the QP is always feasible even in the presence of control input constraints.
    \item Conditions for continuity of the control input as the optimal solution of the QP are studied under milder conditions as compared to prior work, and it is shown that the closed-loop solutions are uniquely determined so that forward invariance of the safe set can be established. 
\end{itemize}
Compared to the earlier literature, the contributions of this paper are summarized as follows. The QP-based approaches in the prior literature, e.g., \cite{li2018formally,srinivasan2018control,nguyen2016exponential,wang2016safety,wang2017safety,ames2014rapidly,ames2017control,lindemann2019control,garg2019control}, do not provide feasibility guarantees for the underlying QP in the presence of input constraints. However, without feasibility of the QP, it is not guaranteed that a control input can be always synthesized, and without consideration of input bounds, the resulting input might not be realizable on a real-world platform. In comparison to these prior studies, we consider control input constraints in addition to the safety and convergence requirements, and guarantee the feasibility of the proposed QP.
The proposed approach further advances the results in \cite{srinivasan2018control,li2019finite,li2018formally} in terms of the achieved time of convergence.
Furthermore, we generalize the results of \cite{ames2017control,morris2015continuity,glotfelter2017nonsmooth}, where the regularity properties of the solution of the QP is discussed in the absence of input constraints; we show continuity of the solution of the proposed QP under the presence of input constraints, and under milder regularity assumptions on the CLF, CBF, and the system dynamics, as compared to the aforementioned work. 

\noindent\textbf{Organization}: The rest of the paper is organized as follows. The main problem formulation is discussed in Section \ref{sec: math prelim}. Mathematical preliminaries of safety and fixed-time stability are reviewed in Section \ref{sec: safe FxT conv}. New results on fixed-time stability under input constraints are presented in Section \ref{sec: rel inp bound van}. The main results on QP-based control synthesis are presented in Section \ref{sec: main results}. Section \ref{sec: simulation} presents two numerical case studies and Section \ref{sec: conclusion} concludes the paper with some directions for future work.

\section{Problem Formulation}\label{sec: math prelim}
In the rest of the paper, $\IR$ denotes the set of real numbers and $\mathbb R_+$ denotes the set of non-negative real numbers. We use
$\|\cdot\|_p$ to denote the $p-$norm, and 
$\|\cdot\|$ to denote the Euclidean norm. We write $\partial S$ for the boundary of a closed set $S\in \mathbb R^n$, $\textnormal{int}(S) \coloneqq S\setminus \partial S$ for its interior, and $|x|_S = \inf_{y\in S}\|x-y\|$ for the distance of a point $x\in \mathbb R^n\setminus S$ from the set $S$. We use $\mathcal C^k$ to denote the set of $k$ times continuously differentiable functions. The Lie derivative of a $\mathcal C^1$ function $V:\mathbb R^n\rightarrow \mathbb R$ along a vector field $f:\mathbb R^n\rightarrow\mathbb R^n$ at a point $x\in \mathbb R^n$ is denoted as $L_fV(x) \coloneqq \frac{\partial V}{\partial x} f(x)$. For two vectors $x, y\in \mathbb R^n$, we use $x\leq y$ to represent element-wise inequalities $x_i\leq y_i, i = 1, 2, \ldots, n$. A continuous function $\alpha:\mathbb R_+\rightarrow\mathbb R_+$ is a class-$\mathcal K$ function if it is strictly increasing and $\alpha(0) = 0$. It belongs to $\mathcal K_\infty$ if in addition, $\lim_{r\to \infty}\alpha(r) = \infty$. A function $\kappa:\mathbb R_+\times\mathbb R_+\rightarrow\mathbb R_+$ is a class-$\mathcal{KL}$ function if 1) for all $t\geq 0$, the map $\kappa(r,t)$ belongs to class-$\mathcal K$ and 2) for all $r\geq 0$, the map $\kappa(r,t)$ is decreasing in $t$ with $\kappa(r,t)\to 0$ as $t\to \infty$. A function $V:\mathbb R^n\rightarrow\mathbb R$ is said to be positive definite with respect to a compact set $S\subset \mathbb R^n$ if $V(x)>0$ for $x\notin S$ and $V(x) = 0$ for $x\in \partial S$. We drop the arguments $t, x$ whenever clear from the context.

Next, we present the problem formulation. Consider the nonlinear, control-affine system
\begin{align}\label{cont aff sys}
    \dot x(t) = f(x(t)) + g(x(t))u(x(t)), \quad x(0) = x_0, 
\end{align}
where $x\in \mathbb R^n$ is the state vector, $f:\mathbb R^n\rightarrow \mathbb R^n$ and $g:\mathbb R^n\rightarrow \mathbb R^{n\times m}$ are system vector fields, continuous in their arguments, and $u\in \mathcal U\subset \mathbb R^m$ is the control input vector where $\mathcal U$ is the input constraint set. In addition, consider a safe set $S_S \coloneqq \{x\; |\; h_S(x)\leq 0\}$ to be rendered forward invariant under the closed-loop dynamics of \eqref{cont aff sys}, and a goal set $S_G \coloneqq  \{x\; |\; h_G(x)\leq 0\}$ to be reached by the closed-loop trajectories of \eqref{cont aff sys} in a user-defined fixed time $T_{ud}>0$, where $h_S, h_G:\mathbb R^n\rightarrow\mathbb R$ satisfy the following assumption. 

\begin{Assumption}\label{assum: sets Ss Sg}
The functions $h_S(x),h_G(x)\in \mathcal C^1$, $S_G\bigcap S_S\neq \emptyset$, {the set $S_G$ is compact}, and the sets $S_S$ and $S_G$ have non-empty interiors. Furthermore, the function $h_G$ is proper with respect to set $S_G$, i.e., there exists a class-$\mathcal K_\infty$ function $\alpha_G$ such that $h_G(x)\geq \alpha_G(|x|_{S_G})$, for all $x\notin S_G$.
\end{Assumption} 

\noindent Note that the boundary and the interior of set $S_S$ (and similarly, of the set $S_G$) are given as $\partial S_S \coloneqq \{x\; |\; h_S(x) = 0\}$ and $\textnormal{int}(S_S) \coloneqq \{x\; |\; h_S(x)<0\}$, respectively. Next, we define the notion of fixed-time domain of attraction for a compact set $S\subset \mathbb R^n$: 

\begin{Definition}[\textbf{FxT-DoA}]\label{def:FxT-DoA} 
For a compact set $S\subset\mathbb R^n$, the set $D\subset \mathbb R^{n}$, satisfying $S\subset D$, is a Fixed-Time Domain of Attraction (FxT-DoA) with time $T>0$ for the closed-loop system \eqref{cont aff sys} under $u$, if 
\begin{itemize}
    \item[i)] for all $x(0) \in D$, $x(t) \in D$ for all $t\in [0, T]$, and
    \item[ii)] there exists $T_0\in [0, T]$ such that $\lim_{t\to T_{0}}x(t) \in S$.
\end{itemize}
\end{Definition} 


In words, a FxT-DoA for a set $S\subset \mathbb R^n$ is a set $D\supset S$ such that it is forward-invariant and starting from any point within the set $D$, the system trajectories reach the set $S$ within a fixed-time $T$. We can now state the main problem considered in this paper.

\begin{Problem}\label{P reach S}
Design a control input $u(x) \in \mathcal U \coloneqq \{v\in \mathbb R^m\; |\; A_uv\leq b_u\}$ and compute $D\subset\mathbb R^n$, so that for all $x_0\in D\subseteq S_S$, the closed-loop trajectories $x(t)$ of \eqref{cont aff sys} satisfy $x(t) \in S_S$ for all $t\geq 0$, and $x(T_{ud})\in S_G$, where $T_{ud}>0$ is a user-defined fixed time and $D$ is a FxT-DoA for the set $S_G$.
\end{Problem}

\noindent Input constraints of the form $\mathcal U = \{v\in \mathbb R^m\; |\; A_uv\leq b_u\}$ are very commonly considered in the literature, see e.g., \cite{cortez2019control,ames2017control}. 
In Section \ref{sec: main results}, we present a quadratic program that computes a control input $u$ that solves Problem \ref{P reach S}.
Before that, we present mathematical preliminaries in the next section.

\section{Preliminaries}\label{sec: safe FxT conv}
\subsection{Forward invariance of safe set}
Problem \ref{P reach S} requires that the closed-loop system trajectories of \eqref{cont aff sys} stay in the set $S_S$ at all times, i.e., the set $S_S$ is \textit{forward-invariant}. 
Forward invariance of a set is defined as follows:

\begin{Definition}
A set $S\subset\mathbb R^n$ is forward invariant for the closed-loop system \eqref{cont aff sys} under the effect of a control input $u(x(t))$ if $x_0\in S$ implies that $x(t)\in S$ for all $t\geq 0$.
\end{Definition}

The following result, known as Nagumo's theorem, is adapted from \cite{blanchini1999set} for forward invariance of the set $S_S$ for the control system \eqref{cont aff sys}:

\begin{Lemma}[\textbf{Nagumo's theorem}]\label{Th: Nagumo for inv}
Let $u(x)\in \mathcal U$ be a continuous control input such that the resulting closed-loop trajectories of \eqref{cont aff sys} are uniquely determined in forward time. Then, the set $S_S$ is forward invariant for the closed-loop system \eqref{cont aff sys} if and only if the following holds:
\begin{align}
    L_fh_S(x)+L_gh_S(x)u(x)\leq 0\quad \forall x\in \partial S_S.
\end{align}
\end{Lemma}
The interested reader is referred to \cite[Section 3.1]{blanchini1999set} for a detailed discussion on forward invariance of sets. We make the following assumption to guarantee that the safe set $S_S$ can be rendered forward invariant for \eqref{cont aff sys}.
\begin{Assumption}\label{Assum feas}
For all $x\in \partial S_S$, there exists a control input $u(x)\in \mathcal U$ such that the following condition holds:
\begin{align}\label{eq: safety cond}
    L_fh_S(x)+L_gh_S(x)u(x)\leq 0.
\end{align}
\end{Assumption}

Similar assumptions have been used in literature, either explicitly (e.g., \cite{romdlony2016stabilization}) or implicitly (e.g., \cite{ames2017control}). In this work, we use the following notion of ZCBFs to ensure forward invariance of the safe set $S_S$. The ZCBF is defined by the authors in \cite{ames2017control} as follows:

\begin{Definition}[\textbf{ZCBF}] For the dynamical system \eqref{cont aff sys}, a continuously differentiable function $h:\mathbb R^n\rightarrow \mathbb R$ is called a ZCBF for the set $S_S$ if $h(x)<0$ for $x\in \textnormal{int}(S_S)$, $h(x) = 0$ for $x\in \partial S_S$, and there exists $\alpha\in \mathcal K$, such that 
\begin{align}\label{eq: ZCBF}
    \inf_{u\in \mathcal U}\{L_fh(x)+L_gh(x)u\}\leq \alpha(-h(x)) \quad \forall x\in S_S.
\end{align}

\end{Definition}

\noindent It is easy to see that if $h$ is a ZCBF, then it also satisfies Assumption \ref{Assum feas}. One special case of \eqref{eq: ZCBF} is {
\begin{align}\label{eq: ZCBF spec}
    \inf_{u\in \mathcal U}\{L_fh(x)+L_gh(x)u\}\leq -\rho h(x),
\end{align}
with $\rho> 0$, as discussed in \cite[Remark 6]{ames2017control}. We will use this special form \eqref{eq: ZCBF spec} in our main results to guarantee forward invariance of the safe set $S_S$.}

\subsection{Overview of {fixed}-time stability}
Next, we review the notion of fixed-time stability. The authors in \cite{polyakov2012nonlinear} define  
the origin to be an FxTS equilibrium of \eqref{cont aff sys} if it is Lyapunov stable and $\lim_{t\to T} x(t)=0$ where the time of convergence $T = T(x(0))$ is uniformly bounded for all $x(0)$, i.e., $\sup\limits_{x(0)\in \mathbb R^n}T(x(0)) <\infty$. The following sufficient conditions for FxTS of the origin are adapted from \cite{parsegov2012nonlinear}.

\begin{Lemma}\label{FxTS TH}
Let $V:\mathbb R^n\rightarrow \mathbb R$ be a continuously differentiable, positive definite, radially unbounded function, satisfying{\small\vspace{-15pt}
\begin{align}\label{eq: dot V old}
     \inf_{u\in \mathcal U}\{L_fV(x)+L_gV(x)u\} \leq -\alpha_1V(x)^{\gamma_1}-\alpha_2V(x)^{\gamma_2},
\end{align}}\normalsize 
for all $x\in \mathbb R^n\setminus\{0\}$, with $\alpha_1,\alpha_2>0$, $\gamma_1 = 1+\frac{1}{\mu}$ and $\gamma_2 = 1-\frac{1}{\mu}$ for some $\mu>1$. Then, the origin of \eqref{cont aff sys} is FxTS with $T$ that satisfies $T \leq \frac{\mu\pi}{2\sqrt{\alpha_1\alpha_2}}$.
\end{Lemma}

\noindent Inspired by \cite{lopez2019conditions}, we define a class of CLF for the system \eqref{cont aff sys}, which is used to encode the convergence of the system trajectories {to a compact set $S\subset \mathbb R^n$ within a user-defined, fixed time $T_{ud}$:}

\begin{Definition}[\textbf{FxT CLF-$S$}]\label{def: FxT CLF}
A continuously differentiable function $V:\mathbb R^n\rightarrow\mathbb R$ is called FxT CLF-$S$ for \eqref{cont aff sys} with parameters $\alpha_1,\alpha_2>0, \gamma_1 = 1+\frac{1}{\mu},\gamma_2 = 1-\frac{1}{\mu}$ with $\mu>1$, if $V$ is proper w.r.t. set $S$ and the following holds:{\small \vspace{-20pt}
\begin{equation}\label{fxts h ineq}
    \inf_{u\in \mathcal U}\{L_fV(x)+L_gV(x)u\}\leq -\alpha_1V(x)^{\gamma_1}-\alpha_2V(x)^{\gamma_2},
\end{equation}}\normalsize
for all $x\in \mathbb R^n \setminus S$, where the time of convergence $T$ satisfies $ T\leq \frac{\mu\pi}{2\sqrt{\alpha_1\alpha_2}}\leq T_{ud}$.
\end{Definition}

\section{Fixed-time stability under input constraints}\label{sec: rel inp bound van}
\subsection{Motivating example}
In this section, we review new FxTS conditions from our recent work in \cite{garg2021FxTSDomain}. These conditions are motivated by, and used in, the control synthesis under fixed-time convergence and input constraints, and are essential in guaranteeing the feasibility of the proposed QP. 
Consider, for the sake of illustration, a 1-dimensional control-affine system 
\begin{align}\label{eq: cont affine 1d}
    \dot x = f(x) + g(x)u,
\end{align}
where $f, g:\mathbb R\rightarrow\mathbb R$ are continuous functions. Suppose that the control objective is to drive the closed-loop trajectories of \eqref{eq: cont affine 1d} to a set $S_V \coloneqq \{x\; |\; V(x)\leq 0\}$ within a user-defined time $T_{ud}$, where $V:\mathbb R\rightarrow\mathbb R$ is continuously differentiable, and proper with respect to the set $S_V$. Additionally, consider the input constraints $u_m\leq u\leq u_M$ where $u_m<u_M$. To this end, following the work in \cite{ames2017control,nguyen2016exponential} and using the {FxTS} conditions from Lemma \ref{FxTS TH}, a {QP} can be formulated as follows:
\begin{subequations}\label{QP 1d old FxTS}
\begin{align}
\min_{u} \; \;\quad \frac{1}{2}u^2 &  \\
    \textrm{s.t.} \quad \begin{bmatrix}
    1 \\ -1 
    \end{bmatrix} u  \leq &  \begin{bmatrix}
    u_M\\-u_m
    \end{bmatrix},\\
    L_fV(x) + L_gV(x)u  \leq &-\alpha_1V(x)^{\gamma_1}  -\alpha_2V(x)^{\gamma_2} \label{C2 stab const 1d old},
\end{align}
\end{subequations}
for $x\notin S_V$, where $c>0$, and $\alpha_1, \alpha_2, \gamma_1, \gamma_2$ are chosen as $\alpha_1 = \alpha_2 = \frac{\mu\pi}{2T_{ud}}$, $\gamma_1 = 1+\frac{1}{\mu}$ and $\gamma_2 = 1-\frac{1}{\mu}$ with $\mu>1$. Existence of the solution of \eqref{QP 1d old FxTS} implies existence of a control input under which the closed-loop trajectories of \eqref{eq: cont affine 1d} reach the set $S_V$ within a fixed time $T$ that satisfies $T\leq \frac{\mu\pi}{2\sqrt{\alpha_1\alpha_2}}$; therefore, by setting $\alpha_1 = \alpha_2 = \frac{\mu\pi}{2T_{ud}}$, it follows that $T\leq T_{ud}$, i.e., the convergence will be achieved within the user-define time. The issue with the {QP} in \eqref{QP 1d old FxTS} is that it might not be feasible for all $x\in \mathbb R^n\setminus S_V$ due to the presence of the input constraints. To address the issue of infeasibility of a QP under multiple constraints, the authors in \cite{ames2017control} introduce a slack variable in the {CLF} constraint. Inspired from this, the new {FxTS} conditions are presented next. 

\subsection{New FxTS Lyapunov conditions}
\begin{Lemma}[\hspace{-0.1pt}\cite{garg2021FxTSDomain}]\label{Th: FxTS new}
Let $V:\mathbb R^n\rightarrow \mathbb R$ be a continuously differentiable, positive definite, radially unbounded function, satisfying
\begin{align}\label{eq: dot V new ineq}
\begin{split}
     \inf_{u\in \mathcal U}\{L_fV(x)+L_gV(x)u\} \leq & \  \delta_1V(x)-\alpha_1V(x)^{\gamma_1}\\
     & -\alpha_2V(x)^{\gamma_2},
\end{split}
\end{align}
for all $x\in \mathbb R^n\setminus\{0\}$, with $\alpha_1, \alpha_2>0$, $\delta_1\in \mathbb R$, $\gamma_1 = 1+\frac{1}{\mu}$, $\gamma_2 = 1-\frac{1}{\mu}$ for some $\mu>1$. Then, 
$D\subset\mathbb R^n$ is a FxT-DoA of the origin of \eqref{cont aff sys} with time $T>0$, where
{\small
\begin{align}
    D & = \begin{cases}\; \mathbb R^n; &  r<1,\\
    \left\{x\; |\; V(x)\leq k^\mu\left(\frac{\delta_1-\sqrt{\delta_1^2-4\alpha_1\alpha_2}}{2\alpha_1}\right)^\mu\right\}; & r\geq 1, 
    \end{cases},\label{new FxTS D est}\\
    T & \leq \begin{cases}\frac{\mu\pi}{2\sqrt{\alpha_1\alpha_2}};& r\leq 0,\\
     \frac{\mu}{\alpha_1k_1}\left(\frac{\pi}{2}-\tan^{-1}k_2\right); & 0\leq r<1,\\
     \frac{\mu k}{(1-k)\sqrt{\alpha_1\alpha_2}}; & r\geq 1, 
    \end{cases},\label{new FxTS T est}
\end{align}}\normalsize
where $r = \frac{\delta_1}{2\sqrt{\alpha_1\alpha_2}}$, $0<k<1$, $k_1 = \sqrt{\frac{4\alpha_1\alpha_2-\delta_1^2}{4\alpha_1^2}}$ and $k_2 = -\frac{\delta_1}{\sqrt{4\alpha_1\alpha_2-\delta_1^2}}$. 
\end{Lemma}

In comparison to Lemma \ref{FxTS TH}, Lemma \ref{Th: FxTS new} allows an additional (possibly, positive) term $\delta_1V$ in the upper bound of the time derivative of the Lyapunov function. {The main idea is to use \eqref{eq: dot V new ineq} in place of the constraint (9c),} with the parameters $\alpha_1, \alpha_2, \mu$ chosen such that $\frac{\mu\pi}{2\sqrt{\alpha_1\alpha_2}} = T_{ud}$, and with $\delta_1$ being a free, slack variable so that feasibility of the QP can be guaranteed. Then, the value of $\delta_1$ would dictate the FxT-DoA $D$.
To see how the condition \eqref{eq: dot V new ineq} can be used to guarantee FxTS in the presence of control input constraints, a new QP can be formulated as follows:
\begin{subequations}\label{QP 1d}
\begin{align}
\min_{u, \delta_1} \; \frac{1}{2}u^2 & +\frac{1}{2}\delta_1^2 +c\delta_1\\
    \textrm{s.t.} \quad \quad  \begin{bmatrix}
    1 \\ -1 
    \end{bmatrix} u  \leq &  \begin{bmatrix}
    u_M\\-u_m
    \end{bmatrix}, \label{C1 cont const 1d}\\
    L_fV(x) + L_gV(x)u  \leq & \;\delta_1 V(x)-\alpha_1V(x)^{\gamma_1}  -\alpha_2V(x)^{\gamma_2} \label{C2 stab const 1d},
\end{align}
\end{subequations}
for $x\notin S_V$, where $c>0$, and $\alpha_1, \alpha_2, \gamma_1, \gamma_2$ are chosen similarly as in \eqref{QP 1d old FxTS}. Note that when $\delta_1$ takes negative values, then out of \eqref{new FxTS T est} the bound on the time of convergence satisfies $T\leq \frac{\mu\pi}{2\sqrt{\alpha_1\alpha_2}} = T_{ud}$, and therefore, convergence within the user-defined time $T_{ud}$ can be achieved. Thus, the linear term $c\delta_1$ is introduced in the cost function with $c>0$ in order to penalize non-positive values of $\delta_1$.  Here, the term $\delta_1 V(x)$ in \eqref{C2 stab const 1d} can be thought of as a slack term, allowing for satisfaction of the constraint \eqref{C2 stab const 1d} in the presence of input constraints \eqref{C1 cont const 1d} as shown below.

\begin{Lemma}\label{lemma: qp 1d feas}
For each $x\notin S_V$, there exist $u(x)\in \mathbb R, \delta_1(x)\in \mathbb R$ satisfying \eqref{C1 cont const 1d}-\eqref{C2 stab const 1d}, i.e., the QP \eqref{QP 1d} is feasible for all $x\notin S_V$.
\end{Lemma}
\begin{pf}
Choose any $\bar u(x)\in [u_m, u_M]$ so that $\bar u(x)$ satisfies the input constraints. For $x\notin S_V$ we have that $V(x)>0$ and therefore, $$\bar \delta_1(x) \coloneqq \frac{L_fV(x) + L_gV(x)\bar u(x)+\alpha_1V(x)^{\gamma_1} +\alpha_2V(x)^{\gamma_2}}{V(x)},$$ is well-defined for all $x\notin S_V$. Note that with $u = \bar u(x), \delta_1 = \bar \delta_1(x)$, \eqref{C2 stab const 1d} is satisfied with equality. Thus, the pair $(\bar u(x), \bar \delta_1(x))$ satisfies \eqref{C1 cont const 1d}-\eqref{C2 stab const 1d} for all $x\notin S_V$, and thus, the {QP} \eqref{QP 1d} is feasible for all $x\notin S_V$. \end{pf}

Note that in prior work, e.g. \cite{ames2017control,nguyen2016exponential}, the slack term is used as 
\begin{align*}
    L_fV(x) + L_gV(x)u \leq \delta-cV(x),
\end{align*}
for some $c>0$. While this condition helps guarantee feasibility of the underlying QP, it does not guarantee that the function $V(x)$ reaches its zero sub-level set when $\delta>0$. Per Lemma \ref{Th: FxTS new}, it holds that the system trajectories reach the zero sub-level set of the function $V(x)$ even when $\delta>0$. This is a unique contribution of the new FxTS condition in Lemma \ref{Th: FxTS new}. 

\subsection{Relation of FxT-DoA with input constraints and time of convergence}\label{sec: DoA T U rel}

As mentioned above, the slack term $\delta_1$ is used to guarantee feasibility of the underlying {QP}. 
Now, the relation of this slack term with FxT-DoA, input constraints and the fixed time of convergence is explored. Let us hypothesize that the slack term corresponding to $\delta_1$ in the {QP} \eqref{QP 1d} characterizes the trade-off between the FxT-DoA and the time of convergence for given control input bounds, and between the FxT-DoA and the input bounds for a given fixed time of convergence. In the particular case when $\alpha_1 = \alpha_2 = \alpha$ for some $\alpha>0$, let us examine how the FxT-DoA $D$ in \eqref{new FxTS D est} is affected by the ratio $r^\star(x) = \frac{\delta^\star(x)}{2\alpha}$, where $(u^\star(x),\delta^\star(x))$ is the optimal solution of the QP \eqref{QP 1d}. Define $r_M = \sup_{x\in S}r^\star(x)$. For $r_M<1$, it holds that $D = \mathbb R^n$, which is the largest possible FxT-DoA. For $r_M\geq 1$, it holds that
\begin{align*}
    D & = \left\{x\; |\; V(x)\leq \inf_{z\in S}k^\mu\left(\frac{\delta_1^\star(z)-\sqrt{(\delta_1^\star(z))^2-4\alpha_1\alpha_2}}{2\alpha_1}\right)^\mu\right\} \\
    & =  \left\{x\; |\; V(x)\leq  \inf_{z\in S} k^\mu\left(r^\star(z)-\sqrt{(r^\star(z))^2-1}\right)^\mu\right\}.
\end{align*}
It can be readily verified that $\left(r-\sqrt{r^2-1}\right)$ is a monotonically decreasing function for $r\geq 1$ and therefore, it holds that
\begin{align}\label{eq: DoA rm}
    D & =  \left\{x\; |\; V(x)\leq k^\mu\left(r_M-\sqrt{r_M^2-1}\right)^\mu\right\}.
\end{align}

\begin{figure}[b]
    \centering
        \includegraphics[width=0.95\columnwidth,clip]{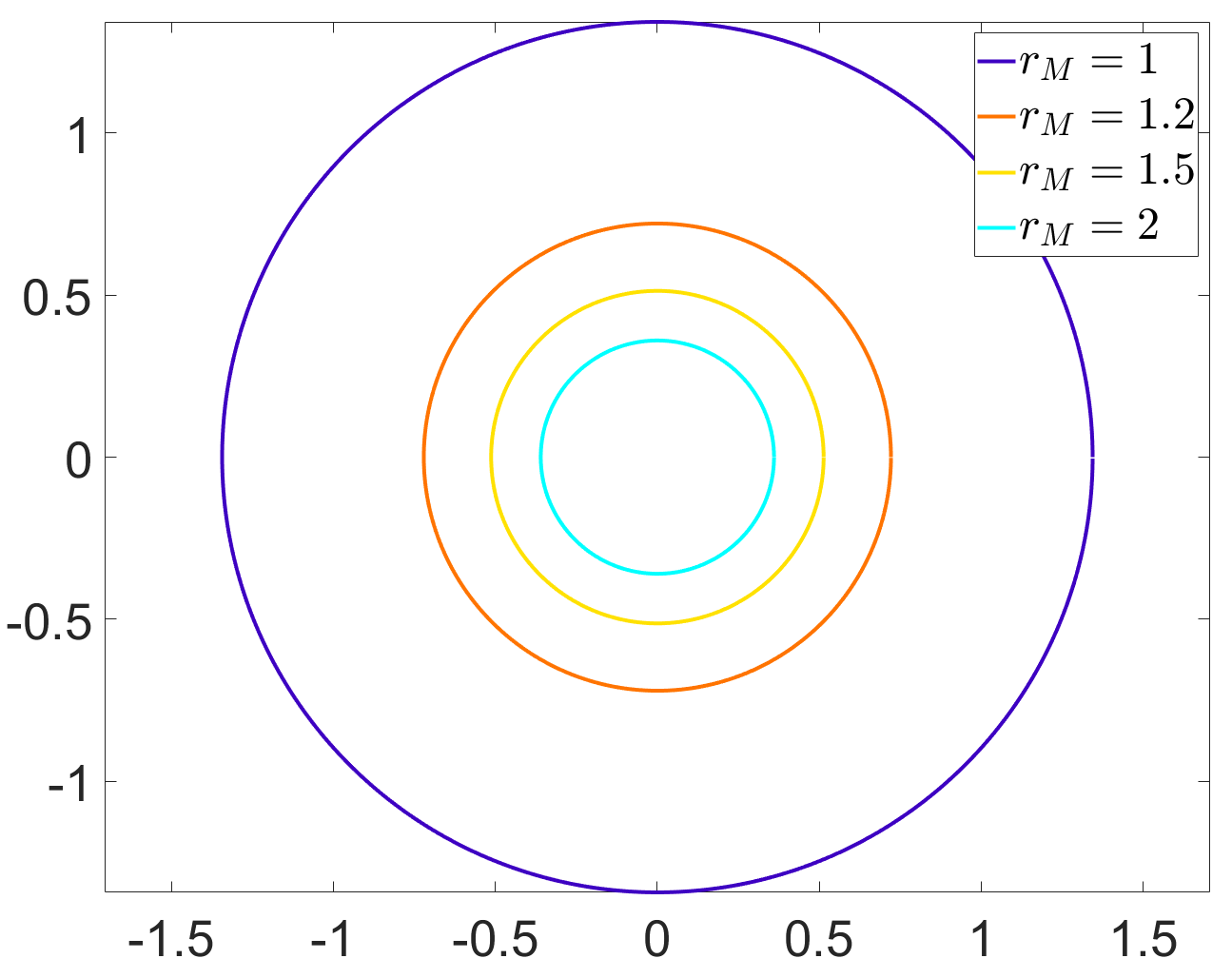}
    \caption{FxT-DoA $D$ for $k = 0.95$ and $\mu = 2$.}\label{fig:DoA rm}
\end{figure}

{Figure \ref{fig:DoA rm} plots the FxT-DoA $D$ for $V(x) = \frac{1}{2}\|x\|^2$. It can be concluded, based on \eqref{eq: DoA rm}, that the domain $D$ shrinks as $r_M$ increases, which is also demonstrated in Figure \ref{fig:DoA rm}. Thus, the larger the FxT-DoA $D$ is, the smaller the value of $r_M$ is. Now, the following Lemma establishes the closed-form expression for the slack variable $\delta_1^\star(x)$ in the particular case when the control input is saturated with $u^\star(x) = u_M$ for the QP \eqref{QP 1d}. 

\begin{Lemma}\label{lemma: delta closed form}
Consider the {QP} \eqref{QP 1d} and let $a(x) \coloneqq L_fV(x) + L_gV(x)u_M+cV(x)+\alpha_1V(x)^{\gamma_1}  +\alpha_2V(x)^{\gamma_2}$. Then, for $x\in S_M$ where $$S_M\coloneqq \{x\; |\; a(x)L_gV(x)+u_MV(x)^2<0, a(x)>0\},$$ the optimal value of $u$ is given as $u^\star(x) = u_M$, and the optimal value of $\delta_1$ is given as
\begin{align}\label{eq: delta closed 1d}
\begin{split}
    \delta_1^\star (x)  = & \frac{L_fV(x)+L_gV(x)u_M}{V(x)} +\alpha_1V(x)^{\gamma_1-1}  \\
    & +\alpha_2V(x)^{\gamma_2-1}.
\end{split}
\end{align}
\end{Lemma}

\noindent The proof is provided in Appendix \ref{app: KKT analysis}. Recall that the {QP} \eqref{QP 1d} is defined for $x\notin S_V$. Note that for $x\in S_M\setminus S_V$, $L_gV(x)<0$ (per definition of the set $S_M$) and $V(x)>0$ (since $V$ is proper w.r.t. the set $S_V$).
Define $ r^\star(x) \coloneqq \frac{\delta_1(x)^\star}{2\sqrt{\alpha_1\alpha_2}}$ so that
\begin{align}\label{eq: r closed 1d}
\begin{split}
    r^\star(x) = & \left(\frac{L_fV(x)+L_gV(x)u_M}{V(x)} \right)\frac{T_{ud}}{\mu\pi}+\frac{1}{2}V(x)^{\gamma_1-1}\\
    &+\frac{1}{2}V(x)^{\gamma_2-1}.
\end{split}
\end{align}
For a given $x\in S_M\setminus S_V$, the expression for the optimal value of $r$ in \eqref{eq: r closed 1d} is a function of the fixed time of convergence $T_{ud}$, and of the upper-bound on the control input $u_M$. Since the region of interest is the one from where the closed-loop trajectories converge to the set $S_V$, consider \eqref{eq: r closed 1d} for the case when $x\in S\coloneqq S_M\cap \{x\; |\; L_fV(x) + L_gV(x)u_M<0\}$. For the restricted domain $S$, it is clear that $r^\star$ decreases as  the control authority increases (i.e., as $u_M$ increases), or the time of convergence requirement relaxes (i.e., as $T_{ud}$ increases). 
Based on this and \eqref{eq: DoA rm}, it is concluded that for a given input bound, relaxing $T_{ud}$ results in a larger FxT-DoA $D$, and conversely, for a given $T_{ud}$, increasing the input bounds results into a larger FxT-DoA. Thus, the hypothesis that the slack variable $\delta_1$ characterizes the trade-off between the FxT-DoA, time of convergence, and the input bounds is verified.}

\section{Main results}\label{sec: main results}
\subsection{QP formulation}
In this section, a control synthesis approach is presented to address Problem \ref{P reach S}. First, a {QP} is designed with guaranteed feasibility, then it is shown that the solution of the {QP} is a continuous function of $x$, and finally, some sufficient conditions are presented under which Problem \ref{P reach S} is solved by the optimal solution of the proposed QP. In what follows, unless specified otherwise, all the results are presented under Assumptions \ref{assum: sets Ss Sg}-\ref{Assum feas}.
Define $z = \begin{bmatrix}v^T & \delta_1 & \delta_2\end{bmatrix}^T\in \mathbb R^{m+2}$, and consider the following {QP}:
\begin{subequations}\label{QP gen}
\begin{align}
\min_{z\in \mathbb R^{m+2}} \; \frac{1}{2}z^THz & + F^Tz\\
    \textrm{s.t.} \quad \; A_uv  \leq & \; b_u, \label{C1 cont const}\\
    L_fh_G(x) + L_gh_G(x)v  \leq & \; \delta_1h_G(x)-\alpha_1\max\{0,h_G(x)\}^{\gamma_1} \nonumber\\
    & -\alpha_2\max\{0,h_G(x)\}^{\gamma_2} \label{C2 stab const}\\
    L_fh_S(x) + L_gh_S(x)v \leq &-\delta_2h_S(x),\label{C3 safe const}
\end{align}
\end{subequations}
where $H = \textrm{diag}\{w_{u_1},\ldots, w_{u_m}, w_1, w_2\}$ is a diagonal matrix consisting of positive weights $w_{u_i}, w_i>0$, $F = \begin{bmatrix}\mathbf 0_m^T & q_1 & 0\end{bmatrix}$ with $q_1>0$ and $\mathbf 0_m\in \mathbb R^m$ a column vector consisting of zeros. The parameters $\alpha_1, \alpha_2, \gamma_1, \gamma_2$ are fixed, and are chosen as $\alpha_1 = \alpha_2 = \frac{\mu\pi}{2T_{ud}}$, $\gamma_1 = 1+\frac{1}{\mu}$ and $\gamma_2 = 1-\frac{1}{\mu}$ with $\mu>1$. The choice of these parameters does not affect the feasibility of the {QP}, as discussed below. In principle, any value of $\mu$ can be chosen as long as it is greater than 1. The linear term $F^Tz = q_1\delta_1$ in the objective function of \eqref{QP gen} penalizes the positive values of $\delta_1$ (see Theorem \ref{Th: d1 d2 P1 solve} for details on why $\delta_1$ being non-positive could be useful). Constraint \eqref{C1 cont const} guarantees that the control input satisfies the control input constraints. Per Lemma \ref{Th: FxTS new}, the constraint \eqref{C2 stab const} guarantees convergence and the constraint \eqref{C3 safe const} ensures safety.

The slack terms corresponding to $\delta_1, \delta_2$ allow the upper bounds of the time derivatives of $h_S(x)$ and $h_G(x)$, respectively, to have a positive term for $x$ such that $h_S(x)<0$ and $h_G(x)>0$. This ensures the feasibility of the {QP} \eqref{QP gen} for all $x$, as demonstrated below.

\begin{Lemma}\label{lemma: QP feasibility}
{Under Assumptions \ref{assum: sets Ss Sg}-\ref{Assum feas}, for each $x\in S_S\setminus S_G$, there exist $v(x)\in \mathbb R^m, \delta_1(x)\in \mathbb R,\delta_2(x)\in \mathbb R$ satisfying \eqref{C1 cont const}-\eqref{C3 safe const}, i.e., the {QP} \eqref{QP gen} is feasible for all $x\in  S_S\setminus S_G$.}
\end{Lemma}

\begin{pf}
Since $x\notin S_G$, it holds that $h_G(x)>0$. Consider the following two cases separately: $h_S(x) = 0$ and $h_S(x)<0$. 

First, let $h_S(x)<0$, i.e., $x\in \textrm{int}(S_S)$. Since $\mathcal U$ is non-empty, there exists $v = \bar v$ in $\mathcal U$ such that $\eqref{C1 cont const}$ is satisfied. Choose $\bar \delta_2 \coloneqq \frac{L_fh_S(x) + L_gh_S(x)\bar v}{-h_S(x)}$, so that \eqref{C3 safe const} is satisfied with equality. 
Also, for $x\in \textrm{int}(S_S)\setminus S_G$, it holds that $h_G(x)>0$. Define $\bar \delta_1 \coloneqq \frac{L_fh_G(x) + L_gh_G(x)\bar v+ \alpha_1h_G(x)^{\gamma_1}+ \alpha_2h_G(x)^{\gamma_2}}{h_G(x)}$, so that \eqref{C2 stab const} holds with equality. Thus, for the case when $h_S(x)<0$, there exists $ (\bar v, \bar \delta_1, \bar \delta_2)$ such that \eqref{C1 cont const}-\eqref{C3 safe const} holds. 

Next, let $h_S(x) = 0$, i.e., $x\in \partial S_S$. Per Assumption \ref{Assum feas}, it holds that that there exists $v = \tilde v\in \mathcal U$ such that \eqref{C3 safe const} holds. Since $h_S(x) = 0$, any value of $\delta_2$ is feasible, and hence, one can choose $\delta_2 = 0$. Hence, the choice of $(v, \delta_1, \delta_2) = (\tilde v, \bar \delta_1, 0)$ satisfies \eqref{C1 cont const}-\eqref{C3 safe const}. Thus, the {QP} \eqref{QP gen} is always feasible. 
\end{pf}


One of the main novelties of the QP \eqref{QP gen} is the way the slack variables are introduced in the FxT-CLF and the ZCBF constraints. Not only these slack variables guarantee that the QP remains feasible under input constraints, they also do not jeopardise forward-invariance of the set $S_S$ or convergence to the goal set $S_G$. In particular, keeping in mind the discussion given after the proof of Lemma \ref{lemma: delta closed form}, the traditional ZCBF condition in the prior work, e.g. \cite{ames2017control,glotfelter2017nonsmooth,glotfelter2018boolean,li2018formally,srinivasan2018control}, uses a particular value of $\delta_2$ for the safety constraint. In the proposed formulation, this parameter is kept as a free variable, so that both safety and feasibility of the QP can be guaranteed simultaneously. 


\subsection{Continuity of the solution of the QP}
Guaranteeing forward invariance of the safe set $S_S$ is based on Lemma \ref{Th: Nagumo for inv}, which in turn requires the uniqueness of the system solutions. Traditionally, Lipschitz continuity of the right-hand side of \eqref{cont aff sys} is utilized in order to guarantee existence and uniqueness of the solutions of \eqref{cont aff sys}, see e.g., \cite{ames2017control,xu2015robustness,lindemann2019control}. 
When the right-hand side of \eqref{cont aff sys} is only continuous, existence and uniqueness of the solutions can be established using the results in \cite[Section 3.15-3.18]{agarwal1993uniqueness} (see Lemma \ref{Thm: exist unique sol QP}). To this end, first, it is shown that the control input $u(x)$ as a solution of the {QP} \eqref{QP gen} is continuous in its arguments. Define $A:\mathbb R^n \rightarrow \mathbb R^{(2m+2)\times (m+2)}$ and $b:\mathbb R^n\rightarrow \mathbb R^{(2m+2)}$ as {\small
\begin{align*}
A(x)& \coloneqq \begin{bmatrix} A_u & {{\mathbf 0_{2m}}} & {{\mathbf 0_{2m}}}\\
L_gh_G(x) & -h_G(x) & 0\\
L_gh_S(x) & 0 & h_S(x)
\end{bmatrix}, \;
b(x)  \coloneqq \begin{bmatrix}b_u \\ b_2(x) \\ -L_fh_S(x)\end{bmatrix}    
\end{align*}}\normalsize
where $b_2(x) \coloneqq -L_fh_G(x)-\alpha_1\max\{0,h_G(x)\}^{\gamma_1}-\alpha_2\max\{0,h_G(x)\}^{\gamma_2}$ and $\mathbf 0_{2m}\in \mathbb R^{2m}$ is a column vector consisting of zeros. Also, define the functions $G_i(x,z) \coloneqq A_i(x)z-b_i(x)$ where $A_i\in \mathbb R^{1\times (m+2)}$ is the $i$-th row of the matrix $A$, and $b_i\in \mathbb R$ the $i$-th element of $b$, so that the constraints \eqref{C1 cont const}-\eqref{C3 safe const} can be written as $G_i(x,z)\leq 0$ for $i\in 1, 2, \ldots, 2m+2$. Let $z^\star $ and $\lambda^\star \in \mathbb R_+^{2m+2}$ denote the optimal solution of \eqref{QP gen}, and the corresponding optimal Lagrange multiplier, respectively. The following assumption is made to prove the main results of this section. 

\begin{Assumption}\label{assum hs hg scs}
The strict complementary slackness holds for \eqref{QP gen} for all $x\in \textnormal{int}(S_S)\setminus S_G$, i.e., for each $i \in \{1, 2, \dots, 2m+2\}$, it holds that either $\lambda_i^\star  > 0$ or $G_i(x,z^\star ) < 0$ for all $x\in \textnormal{int}(S_S)\setminus S_G$.
\end{Assumption}

Complementary slackness, i.e., $\lambda_i^\star  G_i(x,z^\star ) = 0$, for all $i = 1, \ldots, 2m+2$, is a both necessary and sufficient condition for optimality of the solution for {QP}s \cite[Chapter 5]{boyd2004convex}. Note that this condition permits that for some $i$, both $\lambda_i^\star  = 0$ and $G_i(x,z^\star ) = 0$. \textit{Strict} complementary slackness rules out this possibility, and requires that for each $i$, either $\lambda_i^\star $ or $G_i(x,z^\star )$ is non-zero.

\begin{Theorem}\label{Th: QP sol cont}
Under Assumptions \ref{assum: sets Ss Sg} and \ref{assum hs hg scs}, the solution $z^\star:\mathbb R^n\rightarrow\mathbb R^{m+2}$ of \eqref{QP gen} is continuous on $\textnormal{int}(S_S)\setminus S_G$.  
\end{Theorem}

\noindent The proof is provided in Appendix \ref{app: proof Thm QP sol cont}.

\textbf{A note on comparison to earlier work}:
Note that the above result guarantees that the control input defined as $u(x) = v^\star (x)$ is continuous on $\textnormal{int}(S_S)\setminus S_G$. 
Under Assumption \ref{assum hs hg scs}, the authors in \cite{fiacco1976sensitivity} show that the solution is continuously differentiable if the objective function and the constraints functions are twice continuously differentiable. The authors in \cite{ames2017control} assume that the functions $f,g$ and the Lie derivatives $L_fh_S, L_fh_G, L_gh_S, L_gh_G$ are locally Lipschitz continuous to show Lipschitz continuity of the solution of {QP} in the absence of control input constraints. Under similar assumptions, the authors in \cite{glotfelter2017nonsmooth} show that the solution of {QP} is guaranteed to be Lipschitz continuous (in the absence of input constraints) if the {CBF} constraints are inactive, i.e., the constraints are satisfied with strict inequality at the optimal solution $z^\star $, which is same as Assumption \ref{assum hs hg scs}. Note that in the presented formulation, the only requirement is that the functions $f,g$ are continuous, and $h_S, h_G$ continuously differentiable in $x$, which is a relaxation of the prior assumptions. Note also that the authors of \cite{glotfelter2017nonsmooth} extend their results in \cite{glotfelter2018boolean} by utilizing the theory of non-smooth analysis, and strong forward invariance of sets even if the control input is not continuous. Under similar assumptions, the results in \cite{usevitch2020strong} utilize the concept of Clarke tangent cones to guarantee strong forward invariance when the control input is only Lebesgue measurable. 

Next, it is shown that closed-loop trajectories of \eqref{cont aff sys} under $u = v^\star $ exist and are unique. 

\begin{Lemma}\label{Thm: exist unique sol QP}
Let Assumptions \ref{assum: sets Ss Sg}-\ref{assum hs hg scs} hold. 
If the solution of \eqref{QP gen}, given as $z^\star  = [v^\star (\cdot)^T  \; \delta_1^\star (\cdot)\; \delta_2^\star (\cdot)]^T$, satisfies $\mathfrak{d}_1\coloneqq \sup\limits_{\textnormal{int}(S_S)\setminus S_G}\delta_1^\star (x)<\infty$, then there exists a neighborhood $D$ of the set $S_G$ such that the closed-loop trajectory under $u(x) = v^\star (x)$ exists and is unique for all $t\geq 0$ and for all $x(0)\in D$. Furthermore, if $\mathfrak{d}_1\leq 0$, then the result holds with $D = S_S$. 

\end{Lemma}
\begin{pf}
The proof is based on \cite[Theorem 3.18.1]{agarwal1993uniqueness}. Using \cite[Theorem 3.15.1]{agarwal1993uniqueness} and choosing a Lyapunov candidate $\mathfrak{v} = \frac{1}{2}|y|^2$, it can be shown that $y\equiv 0$ is the unique solution of $\dot y = 0$ for $y(0) = 0$. Theorem \ref{Th: QP sol cont} guarantees that the solution of the {QP} \eqref{QP gen} is continuous, which implies continuity of the closed-loop system dynamics \eqref{cont aff sys} when $u(x) \coloneqq v^\star (x)$. Note that $h_G(x) = 0$ for $x\in \partial S_G$ and $h_G(x)>0$ for $x\notin S_G$, i.e., the function $h_G$ is positive definite with respect to the set $S_G$. Define $\phi(y)\coloneqq \mathfrak{d}_1y-\alpha_1\sign(y)|y|^{\gamma_1}-\alpha_2\sign(y)|y|^{\gamma_2}$. Per Lemma \ref{Th: FxTS new}, it holds that there exists a neighborhood $D_y\subset\mathbb R$ of the origin such that for all $y\in D_y$, $\phi(y)\leq 0$. Thus, there exists a function $g$ defined as $g(t,V) = 0$ that satisfies condition (i) of \cite[Theorem 3.18.1]{agarwal1993uniqueness}, the closed-loop dynamics of \eqref{cont aff sys} satisfies the condition (ii), and there exists a function $V$ defined as $V(x) = h_G(x)$ that satisfies the condition (iii).
Thus, using \cite[Theorem 3.18.1]{agarwal1993uniqueness}, there exists $\tau>0$ such that the solution of the closed-loop system \eqref{cont aff sys} exists and is unique for all $x(0)\in D = \{x\; |\; V(x)\in D_y\}$ and all $t\in [0, \tau)$. Since the closed-loop solution $x(t)$ is bounded in the compact set $D$, the solution is complete (see \cite[Ch2., Theorem 1]{aubin2012differential}), and thus, $\tau = \infty$. 

Finally, in the case when $\mathfrak d_1\leq 0$, it holds that the $D_y = \mathbb R$, and thus, the result holds with $D = S_S$. 
\end{pf}

\subsection{Safety and fixed-time convergence}
\noindent Finally, it is shown that under some conditions, the solution of \eqref{QP gen} solves Problem \ref{P reach S}. 

\begin{Theorem}\label{Th: d1 d2 P1 solve}
Let Assumptions \ref{assum: sets Ss Sg}-\ref{assum hs hg scs} hold. If the solution of \eqref{QP gen}, given as $z^\star  = [v^\star (\cdot)^T  \; \delta_1^\star (\cdot)\; \delta_2^\star (\cdot)]^T$, satisfies $\delta_1^\star (x) \leq 0, \ \ \forall x\in \textnormal{int}(S_S)$, then, $u(\cdot) = v^\star (\cdot)$ solves Problem \ref{P reach S} for all $x(0)\in \textnormal{int}(S_S)$, i.e., $D = \textnormal{int}(S_S)$. 
\end{Theorem}
\begin{pf}
First, the convergence of the closed-loop trajectories $x(t)$ to the set $S_G$ within the user-defined time $T_{ud}$ is shown.  Since $\delta_1^\star (x)\leq 0$, per Lemma \ref{Th: FxTS new}, it holds that the closed-loop trajectories of \eqref{cont aff sys} with $u(x) \coloneqq v^\star (x)$ reach the set $S_G$ within fixed time $T\leq \frac{\mu\pi}{2(\alpha_1 \alpha_2)^{\frac{1}{2}}} = T_{ud}$, i.e., within the user-defined time $T_{ud}$ for all $x(0)\in \textnormal{int}(S_S)$. 

Next, it is shown that the closed-loop trajectories of \eqref{cont aff sys} satisfy $x(t)\in \textnormal{int}(S_S)$ for all $t\leq T_{ud}$ under $u(x) \coloneqq v^\star (x)$. From Lemma \ref{Thm: exist unique sol QP}, it holds that the closed-loop solution of \eqref{cont aff sys} exists and is unique under $u = v^\star $ for all $0\leq t\leq T_{ud}$ and for all $x(0)\in \textnormal{int}(S_S)$. Using the similar arguments as in \cite[Theorem 1]{ames2017control}, it can be shown that the set $\textnormal{int}(S_S)$ is forward-invariant. Therefore, the control input $u(x) \coloneqq v^\star (x)$ solves Problem \ref{P reach S} for all $x(0)\in \textnormal{int}(S_S)$. 
\end{pf}

\begin{Remark}
As pointed out in \cite{ames2017control}, the conflict between safety and the convergence constraint require a non-zero slack term for satisfaction of \eqref{C2 stab const}-\eqref{C3 safe const} together. With this observation and keeping in mind the discussion in Section \ref{sec: rel inp bound van}, one can readily conclude that if the control-input bounds or the user-defined time $T_{ud}$ is sufficiently large, then it is possible to satisfy \eqref{C2 stab const} with $\delta_1\leq 0$. 
\end{Remark}

\begin{Remark}
Utilizing the notion of {FxT CLF} from Definition \ref{def: FxT CLF}, it can be shown that the function $h_G$ being {FxT} {CLF}-$S_G$ is a sufficient condition for the solution of the QP \eqref{QP gen} to generate a control input that solves Problem \ref{P reach S}. In particular, if $h_G$ is an {FxT} {CLF}-$S_G$ for \eqref{cont aff sys} with parameters $\alpha_1, \alpha_2, \gamma_1, \gamma_2$ as defined in \eqref{QP gen}, then the {QP} \eqref{QP gen} with $\delta_1 = 0$ is feasible {for all $x\in \textnormal{int}(S_S)\setminus S_G$}, i.e., the solution $(v^\star (x),\delta_2^\star (x))$ of the {QP} \eqref{QP gen} after fixing $\delta_1 = 0$ exists for all $x\in \textnormal{int}(S_S)\setminus S_G$. 
\end{Remark}

Next, some cases are listed when the solution of {QP} \eqref{QP gen} might not solve Problem \ref{P reach S} with the specified time constraint and from all initial conditions, but it still renders the closed-loop trajectories safe, and convergent to the set $S_G$ within some fixed time. 


\begin{Theorem}\label{Th: delta a1 a2 Th 2}
Under Assumptions \ref{assum: sets Ss Sg}-\ref{assum hs hg scs}, the following hold: 
\begin{itemize}
    \item[1)] If the solution of \eqref{QP gen} satisfies {\small
\begin{align}\label{eq: delta1 a1 a2 cond 2}
   \mathfrak{d}_1 \coloneqq \sup_{x\in S_S\setminus S_G}\delta_1^\star(x)& < 2\sqrt{\alpha_1\alpha_2} = \frac{\mu\pi}{T_{ud}},
\end{align}}\normalsize
then, for all $x(0)\in \textnormal{int}(S_S)$, the closed-loop trajectories $x(t)$ of \eqref{cont aff sys} under $u(x)\coloneqq v^\star(x)$ reach the set $S_G$ in a fixed time $$T_1\leq \frac{\mu}{\alpha_1\mathrm{k}_1}\left(\frac{\pi}{2}-\tan^{-1}\mathrm{k}_2\right),$$ where $\mathrm{k}_1 \coloneqq \sqrt{\frac{4\alpha_1\alpha_2-\mathfrak{d}_1}{4\alpha_1^2}}$ and $\mathrm{k}_2 \coloneqq -\frac{\mathfrak{d}_1}{\sqrt{4\alpha_1\alpha_2-\mathfrak{d}_1}}$, while satisfying $x(t)\in S_S$ for all $t\geq 0$, i.e., $D = \textnormal{int}(S_S)$. 
\item[2)] If \eqref{eq: delta1 a1 a2 cond 2} does not hold, then there exists $D\subset \textnormal{int}(S_S)$ such that for all $x(0)\in D$, the closed-loop trajectories satisfy $x(t)\in \textnormal{int}(S_S)$ for all $t\geq 0$ and reach the goal set $S_G$ within a fixed time $$T_2\leq  \frac{\mu k}{(1-k)\sqrt{\alpha_1\alpha_2}},$$ where $0<k<1$, where $D$ is the largest sub-level set of the function $h_G$ in the set $D_G \cap \textnormal{int}(S_S)$, with $D_G = \{x\; | \; h_G(x)\leq k^\mu\left(\frac{\mathfrak{d}_1-\sqrt{\mathfrak{d}_1-4\alpha_1\alpha_2}}{2\alpha_1}\right)^\mu\}$.
\end{itemize}
\end{Theorem}

\begin{figure}[t]
    \centering
    \includegraphics[width=0.9\columnwidth,clip]{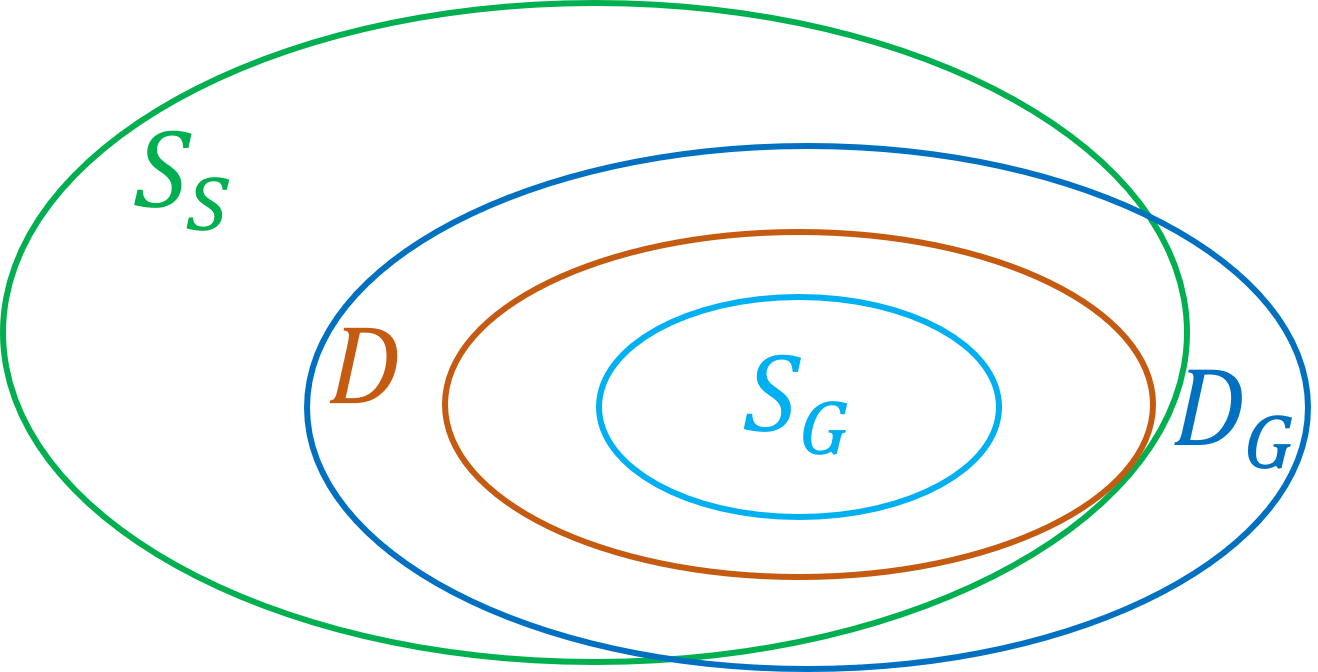}
    \caption{Illustration of the safe set $S_S$ (shown in green), the goal set $S_G$ (shown in light blue), FxT-DoA $D_G$ for the goal set (shown in dark blue) and the domain $D$ (shown in brown).}\label{fig:set DS}
\end{figure}

\begin{pf}
In both cases, following the proof of Theorem \ref{Th: d1 d2 P1 solve}, it holds that the closed-loop trajectories satisfy $x(t)\in \textnormal{int}(S_S)$ for all $0\leq t\leq T$ for any $T<\infty$. When \eqref{eq: delta1 a1 a2 cond 2} holds, using Lemma \ref{Th: FxTS new}, it follows that the closed-loop trajectories $x(t)$ of \eqref{cont aff sys} under $u(x)\coloneqq v^\star (x)$ reach the set $S_G$ within fixed time $T_1$ for all $x(0)\in \textnormal{int}(S_S)$ satisfying $T_1\leq \frac{\mu}{\alpha_1\mathrm{k}_1}\left(\frac{\pi}{2}-\tan^{-1}\mathrm{k}_2\right)$, where $\mathrm{k}_1 \coloneqq \sqrt{\frac{4\alpha_1\alpha_2-\mathfrak{d}_1}{4\alpha_1^2}}$ and $\mathrm{k}_2 \coloneqq -\frac{\mathfrak{d}_1}{\sqrt{4\alpha_1\alpha_2-\mathfrak{d}_1}}$. Also, per \eqref{eq: delta1 a1 a2 cond 2}, it holds that $\mathrm{k}_1>0$ and so, $T_1<\infty$. 

For the case when \eqref{eq: delta1 a1 a2 cond 2} does not hold, using Lemma \ref{Th: FxTS new}, it holds that the closed-loop trajectories of \eqref{cont aff sys} under $u(x)\coloneqq v^\star (x)$ reach the set $S_G$ within time $T_2$ for all $x(0)\in D_G$ where $T_2\leq  \frac{\mu k}{(1-k)\sqrt{\alpha_1\alpha_2}}$ and $D_G \coloneqq \left\{x\; | \; h_G(x)\leq k^\mu\left(\frac{\mathfrak{d}_1-\sqrt{\mathfrak{d}_1-4\alpha_1\alpha_2}}{2\alpha_1}\right)^\mu\right\}$, where $0<k<1$. Since it is also required that $x(0)\in \textnormal{int}(S_S)$, define $D$ as the largest sub-level set of the function $h_G$ in the set $D_G\cap \textnormal{int}(S_S)$, so that $D$ is forward invariant (see Figure \ref{fig:set DS}). Therefore, for all $x(0)\in D$, the closed-loop trajectories of \eqref{cont aff sys} reach the set $S_G$ within the fixed time $T_2$, while maintaining safety at all times.
\end{pf}

In brief, the solution of the {QP} \eqref{QP gen} always exists, is a continuous function of $x$, and renders the set $\textnormal{int}(S_S)$ forward invariant, i.e., guarantees safety. Furthermore, the control input is guaranteed to yield fixed-time convergence of the closed-loop trajectories to the goal set $S_G$. In the case when $\delta_1\leq 0$, the convergence is guaranteed for all $x(0)\in \textnormal{int}(S_S)$, and within the user-defined {fixed} time $T_{ud}$. If $\delta_1$ satisfies \eqref{eq: delta1 a1 a2 cond 2}, then fixed-time convergence is guaranteed for all $x(0)\in \textnormal{int}(S_S)$ (i.e., $D = \textnormal{int}(S_S)$), but the time of convergence $T_1$ may exceed the time $T_{ud}$. Finally, if \eqref{eq: delta1 a1 a2 cond 2} does not hold, then fixed-time convergence is guaranteed for all $x(0)\in D\subset \textnormal{int}(S_S)$, however, the time of convergence $T_2$ may exceed the time $T_{ud}$. 

\section{Numerical Case Studies}\label{sec: simulation}
We present two case studies to illustrate the efficacy of the proposed method. 
{We use Euler discretization to discretize the continuous-time dynamics, and the \textsc{Matlab} function $\texttt{quadprog}$ to solve the QP at each discrete time step. }

\subsection{Adaptive Cruise Control Problem}
In this example, we consider an adaptive cruise control (ACC) problem with a following and a lead vehicle, where the objective for the following vehicle is to achieve a desired speed, and maintain a safe distance from the lead vehicle. Considering that the two vehicles are modeled as point masses and travelling along a straight line, the system dynamics can be written as $\dot x = f(x) + gu$ with
$f(x) = \begin{bmatrix}
-F_r(x)/M & a_L & x_2-x_1
\end{bmatrix}^T, g=  
\begin{bmatrix}
1/M & 0 & 0
\end{bmatrix}^T$, 
where $u \in [-u_{\max}, u_{\max}]$ is the control input, $x = [x_1, x_2, x_3]^T = [v_f, v_l, D]^T \in \IR^3$ is the system state with $v_f$ being the velocity of the following vehicle, $v_l$ being the velocity of the lead vehicle, and $D$ being the distance between the two vehicles (see \cite{ames2017control} for more details). Here, $M$ is the mass of the following vehicle, $F_r(x) = f_0 + f_1 v_f + f_2 v_f^2$ is the drag force, and $a_L \in (-a_l \mathbf a_g, a_l \mathbf a_g)$ is the acceleration of the lead vehicle, with $a_l$ being the fraction of the gravitational acceleration $\mathbf a_g$.

\begin{figure}[!t]
    \centering
    \includegraphics[ width=1\columnwidth,clip]{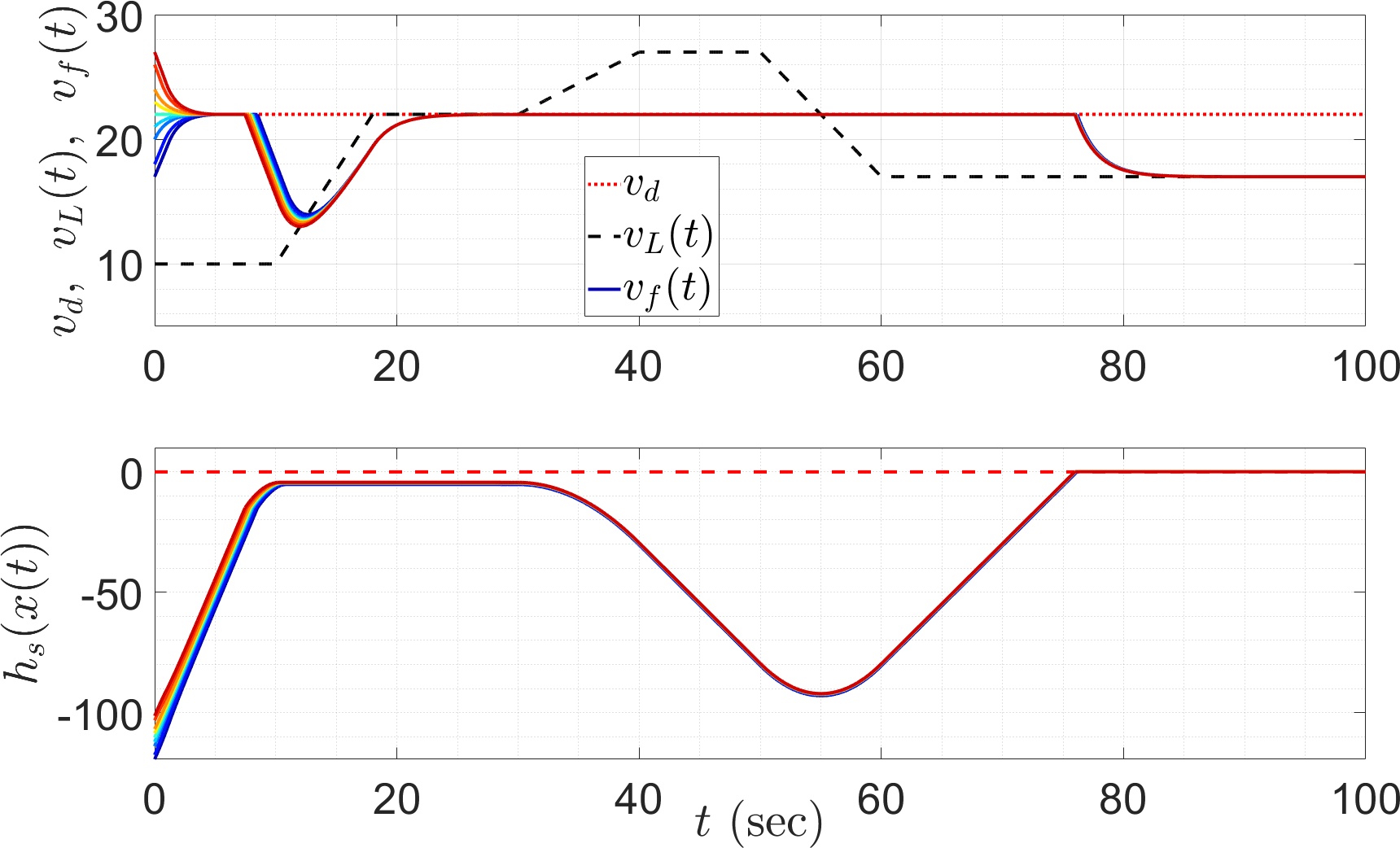}
    \caption{Tracking Performance and the safe set $h_S(x)$ for various initial follower velocities $v_f(0) \in [17,27]$ $\mathrm{m/s}$ with $T=10$ sec.}\label{ex1_1}
\end{figure}

We define the goal and the safe sets, respectively, using the functions $h_G(x) = (v_f - v_d)^2, h_S(x) = \tau_d v_f - D,$ where $v_d$ is a desired fixed velocity and $\tau_d$ is the desired time headway. We set the maximum available control effort to $u_{\max}=0.25 M\mathbf a_g$ with $\mathbf a_g=9.81$ $\mathrm{m/s^2}$ and $M=1650$ $\mathrm{Kg}$, the desired velocity to $v_d=22 $ $\mathrm{m/s}$, the initial velocity of the lead vehicle to $v_l(0)=10$ $\mathrm{m/s}$, initial distance to $D(0) = 150$ $\mathrm{m}$, $f_0=0.1$ $\mathrm{N}$, $f_1=5$ $\mathrm{Ns/m}$, $f_2=0.25$ $\mathrm{Ns^2/m^2}$, and $a_l = 0.3$. We implement the QP in \eqref{QP gen} with $T_{ud}=10$ sec, and $\mu=5$ resulting in $\gamma_1=1.2$, $\gamma_2=0.8$.

Figures \ref{ex1_1} and \ref{ex1_2} illustrate the tracking performance of the resulting controller, where the solid lines represent the velocity of the following vehicle for different initial velocity of the following vehicle $v_f(0) \in [17, 27]$ $\mathrm{m/s}$. The desired speed is achieved when the trajectories are away from the boundaries of the safe set, while closer to the boundaries of the safe set the speed of the following vehicle is reduced to maintain safety. 

\begin{figure}[t]
    \centering
   \includegraphics[ width=1\columnwidth,clip]{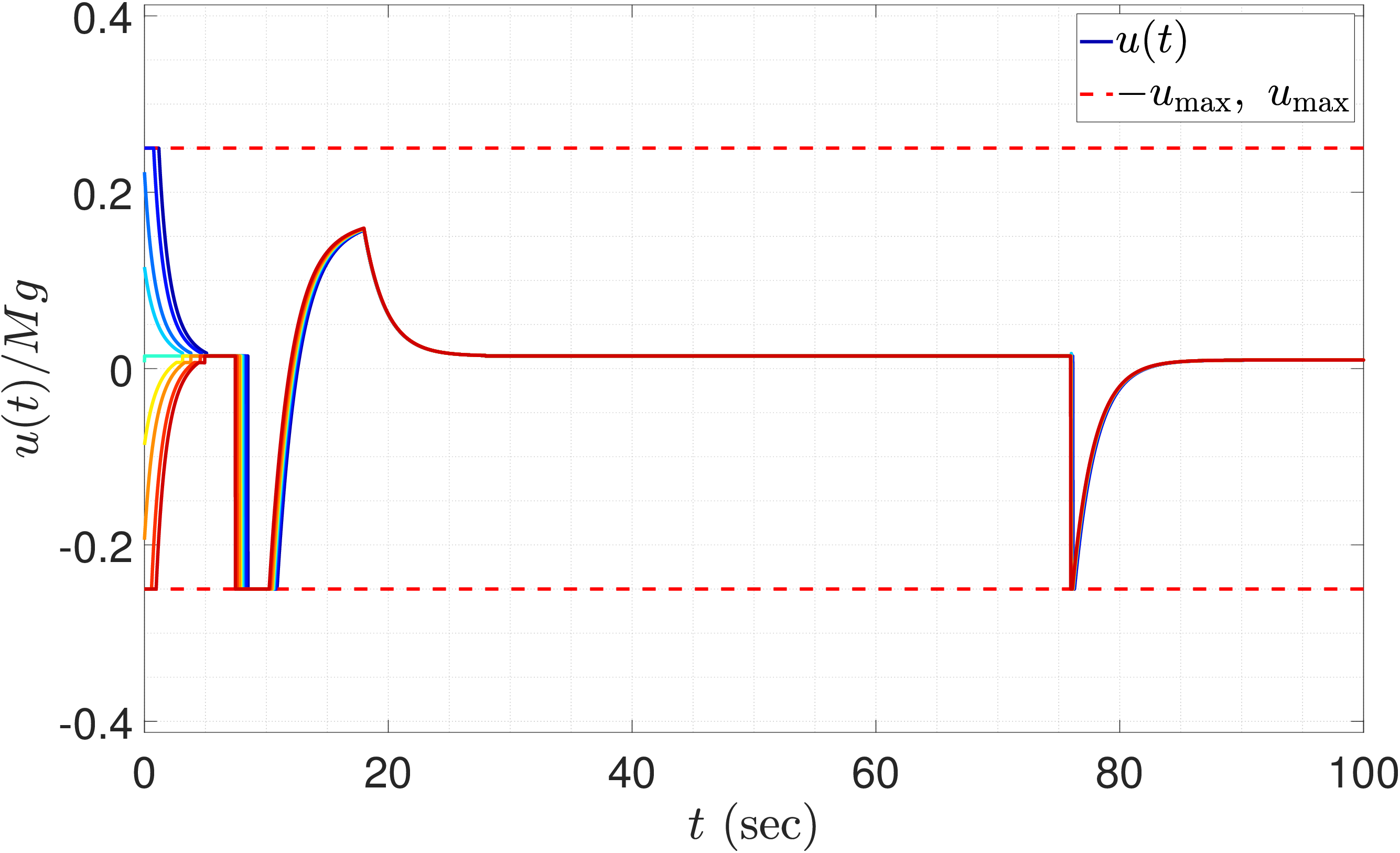}
        \caption{Control input for the closed-loop trajectories for various $v_f(0)$.}\label{ex1_2}
\end{figure}

\begin{figure}[!t]
    \centering
    \includegraphics[ width=1\columnwidth,clip]{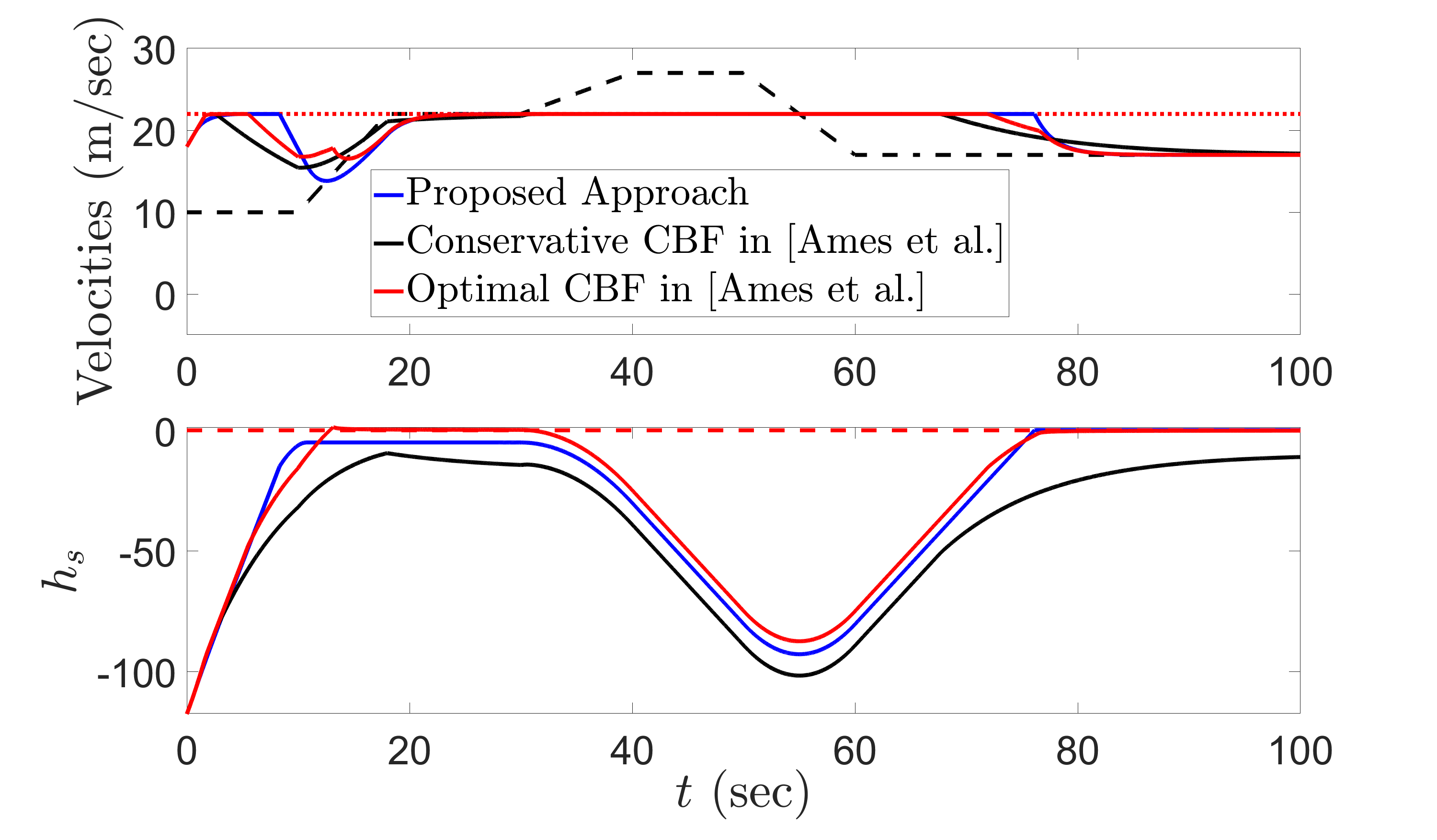}
    \caption{Tracking performance comparison of the proposed approach and the results in \cite{ames2017control} (referenced as [Ames et al.]). The reference velocity is shown in black dotted line and the lead vehicle's velocity in red dotted line. }\label{ex1_0} 
\end{figure}

\begin{figure}[t]
    \centering
    \includegraphics[ width=1\columnwidth,clip]{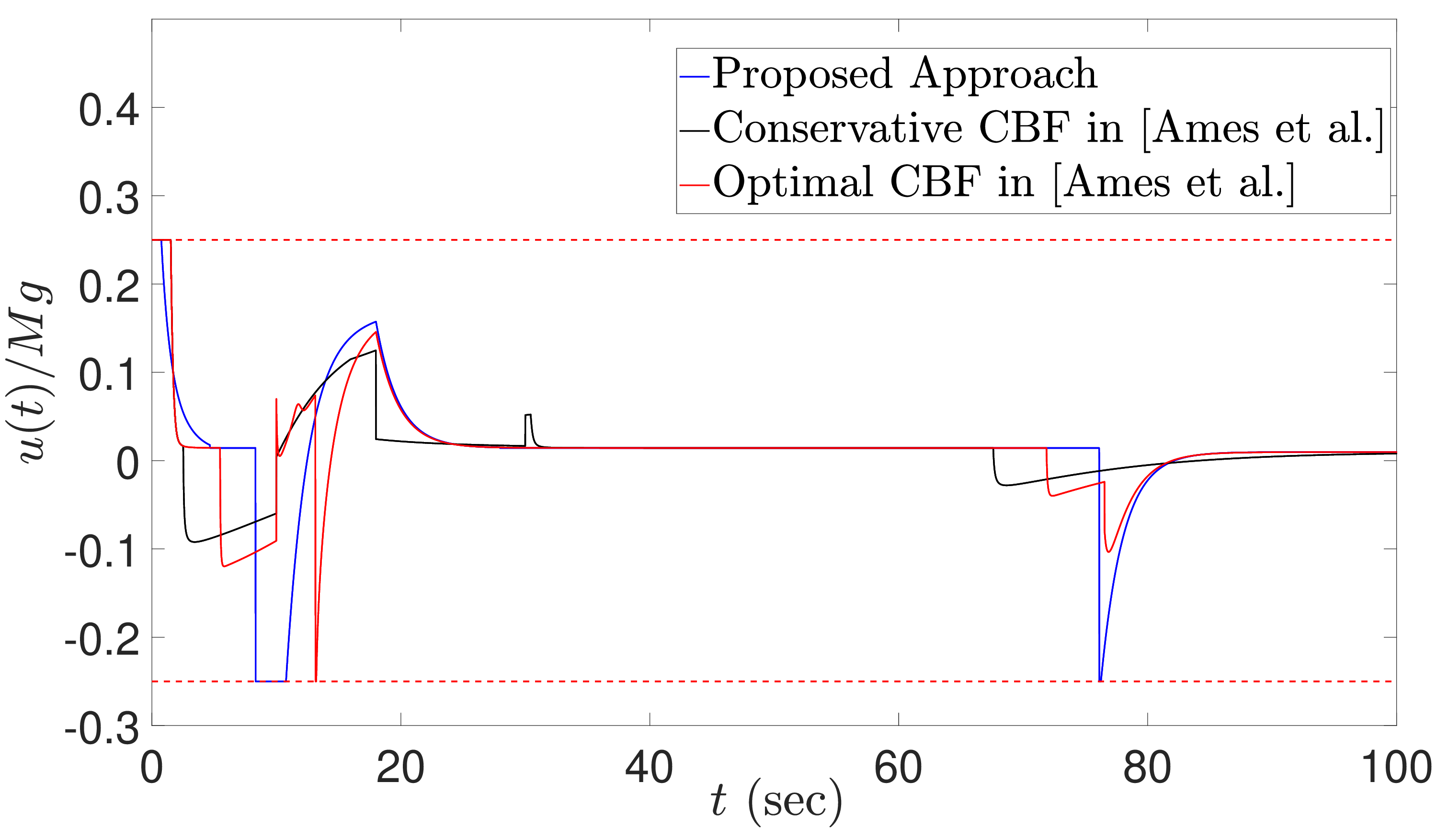}
    \caption{Control input for the proposed approach and the results in \cite{ames2017control}.}\label{ex1_00}
\end{figure}

As stated before, there is no guarantee for the existence of the solution of the  proposed QP in \cite{ames2017control} when there is a control input constraint. For the specific problem of adaptive cruise control as in this  example, the authors in \cite{ames2017control} introduced two control barrier functions, namely optimal and conservative CBFs, based on the simplified system dynamics with no drag effect $F_r(x)$ to ensure feasibility of the solution. However, due to conservatism, the newly constructed safe sets $h_F^o(x)$ and $h_F^c(x)$ for the  optimal and conservative CBFs are violated initially for large initial velocity of the following vehicle, while the actual safe set $h_S(x) = \tau_d v_f - D$ is not violated and the problem can be still feasible. 

\begin{figure}[t] 
    \centering
    \includegraphics[ width=1\columnwidth,clip]{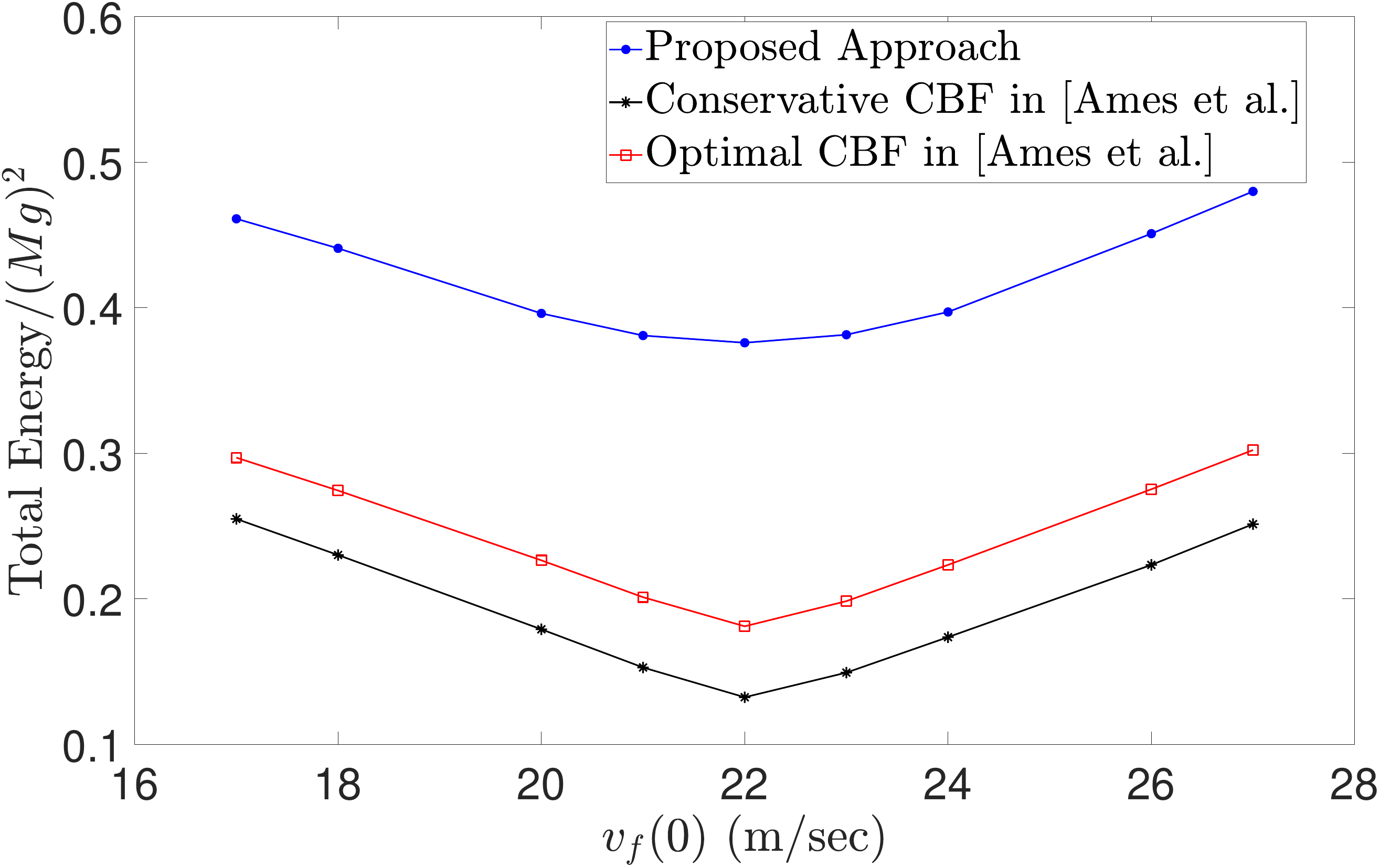}
    \caption{Energy for the proposed approach and the results in \cite{ames2017control}.}\label{ex1_3}
\end{figure}

Figures \ref{ex1_0} and \ref{ex1_00} compare  the tracking performance of the proposed approach and the results with optimal and conservative CBFs with $v_f(0) = 18$ $\mathrm{m/s}$. 
Since we are solving the QP directly and without the aforementioned conservatism, one can see from Figure \ref{ex1_0} that our proposed control approach tracks the desired goal speed of $22$ $\mathrm{m/s}$ for a longer duration before departure from this speed for maintaining safety.

Finally, Figure \ref{ex1_3} compares the control effort between the the proposed approach and the results in \cite{ames2017control}, where the proposed approach is using more available control authority. This is due to the fact that the desired goal speed is tracked for a longer duration in the proposed approach, and hence more control action is used to keep the system trajectories in the safe set as the trade-off. 

\subsection{Motion planning with spatiotemporal specifications}\label{sec sim B}
In the second scenario, we present a two-agent motion planning example under spatiotemporal specifications, where the robot dynamics are modeled under constrained unicycle dynamics as 
\begin{subequations}\label{eq: unicycle dyn}
\begin{align}
    \dot x_i & = u_i\cos(\theta_i), \\
    \dot y_i & = u_i\sin(\theta_i), \\
    \dot \theta_i & = \omega_i,
\end{align}
\end{subequations}
where $[x_i\; \ y_i]^T\in \mathbb R^2$ is the position vector of the agent $i$ for $i\in \{1, 2\}$, $\theta_i\in \mathbb R$ its orientation and $[u_i\; \ \omega_i]^T\in \mathbb R^2$ the control input vector comprising of the linear speed $u_i\in [0, u_M]$ and angular velocity $|\omega_i|\leq \omega_M$. The closed-loop trajectories for the agents, starting from $[x_1(0)\; \ y_1(0)]^T\in C_1 = \{z\in \mathbb R^2\; |\; \|z-[-1.5\; 1.5]^T\|_\infty\leq 0.5\}$ and $[x_2(0)\; \ y_2(0)]^T\in C_2 = \{z\in \mathbb R^2\; |\; \|z-[1.5\; 1.5]^T\|_\infty\leq 0.5\}$, respectively, are required to reach to sets $C_2$ and $C_1$, while staying inside the blue rectangle $\{z\in \mathbb R^2\; |\; \|z\|_\infty\leq 2\}$, and outside the red-dotted circle $\{z\in \mathbb R^2\; |\; \|z\|_2 \leq 1.5\}$, as shown in Figure \ref{scene:2}. The agents are also required to maintain a minimum inter-agent distance $d_m>0$ at all times. 

\begin{figure}[t]
    \centering
    \includegraphics[ width=1\columnwidth,clip]{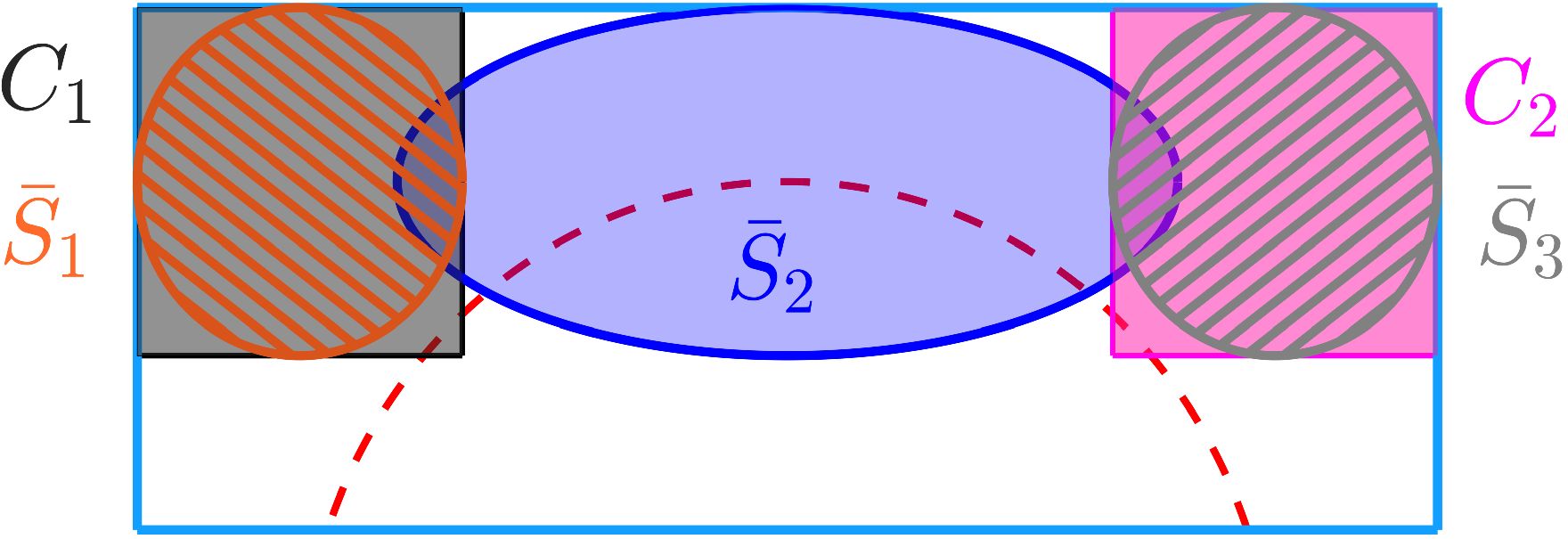}
    \caption{Problem setting and sets $\bar S_1, \bar S_2, \bar S_3$ for simulation example 2.}\label{scene:2}
\end{figure}

Note that the sets $C_i$ are not overlapping with each other, and the corresponding functions $h_i(x)$ are not continuously differentiable. 
Thus, to satisfy Assumption \ref{Assum feas} and use the QP \eqref{QP gen}, {we construct auxiliary sets $\bar S_1 = \{[x\; \ y]\; |\; \frac{(x+1.5)^2}{0.5^2}+\frac{(y-1.5)^2}{0.5^2}\leq 1\}$ (orange circle), $\bar S_2 = \{[x\; \ y]^T\; |\;\frac{x^2}{1.2^2}+\frac{(y-1.5)^2}{0.5^2}\leq 1\}$ (blue ellipse) and  $\bar S_3 = \{[x\; \ y]^T\; |\; \frac{(x-1.5)^2}{0.5^2}+\frac{(y-1.5)^2}{0.5^2}\leq 1\}$ (grey circle) as shown in Figure \ref{scene:2}}.


\begin{figure}[b]
    \centering
    \includegraphics[ width=1\columnwidth,clip]{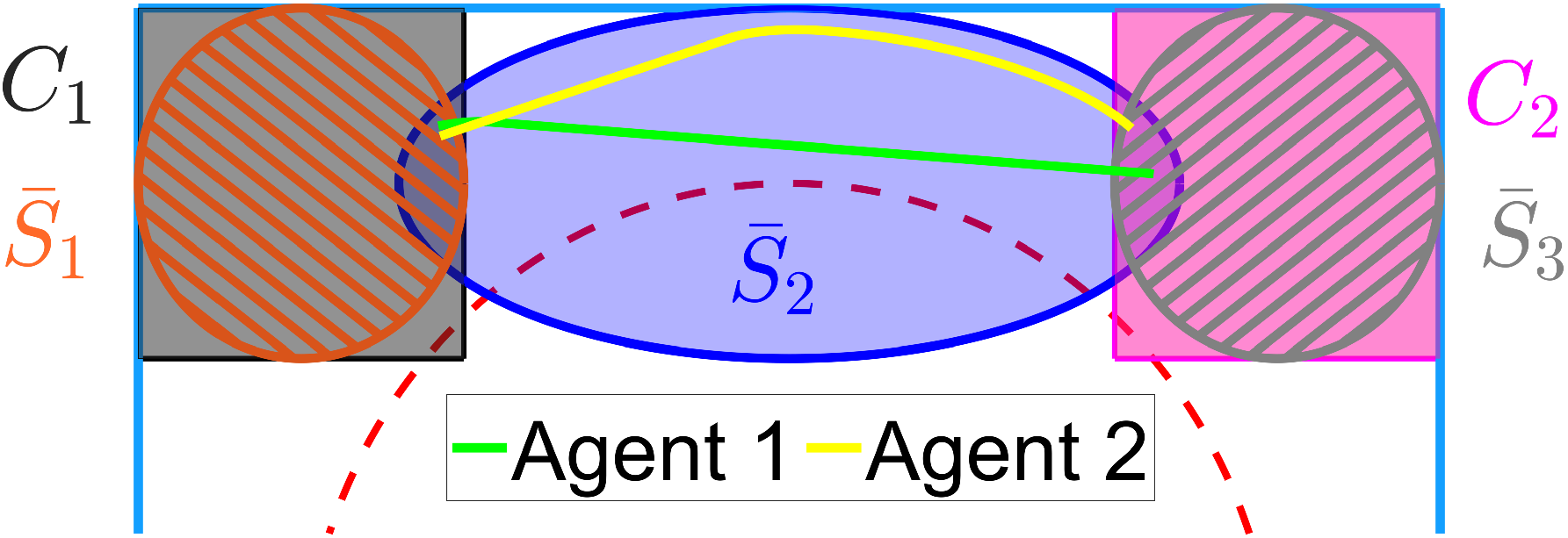}
    \caption{The resulting closed-loop paths of the agents.}\label{fig: ex3 traj}
\end{figure}

{We choose the barrier functions as
\begin{subequations}
\begin{align}
    h_{ij}(x_1,x_2,y_1,y_2) & = d_m^2-\|[x_1-x_2\; \ y_1-y_2]^T\|^2\\
    h_1(x,y) & = 1.5^2 - \|[x\; \ y]^T\|^2, \\
    h_2(x,y) & = [x\; \ y - 1.5]P[x\; \ y-1.5]^T-1, \\
    h_\theta(\theta) & = 0.1^2-(\theta-\phi-\pi)^2,
\end{align}
\end{subequations}
where $P = \begin{bmatrix}1/1.5^2 & 0 \\ 0 & 1/0.5^2\end{bmatrix}$ and $\phi \triangleq \angle([x\; \ y]^T)$ is the angle of the position vector $[x\; \ y]^T$ from the $x-$axis. The functions $h_1$ and $h_2$, along with $h_\theta$, help keep the agent inside the set $\bar S_2$ and outside the red-dotted circle, respectively, in Figure \ref{scene:2}. We choose the Lyapunov function as
\begin{subequations}
\begin{align}
    V_1 & = (x-x_g)^2 +(y - y_g)^2-0.5^2,\\
    V_\theta & = (\theta-\phi_g)^2 - 0.1^2,
\end{align}
\end{subequations}
where $[x_g\; \ y_g]^T = [1.5\; \ 1.5]^T$ is the goal location and $\phi_g = \angle([x_g\; \ y_g]^T- [x\; \ y]^T)$ is the angle between the x-axis and the vector that is defined from the agent's location to the goal point. These functions help steer the agent towards the goal location.

\begin{figure}[t]
    \centering
    \includegraphics[ width=1\columnwidth,clip]{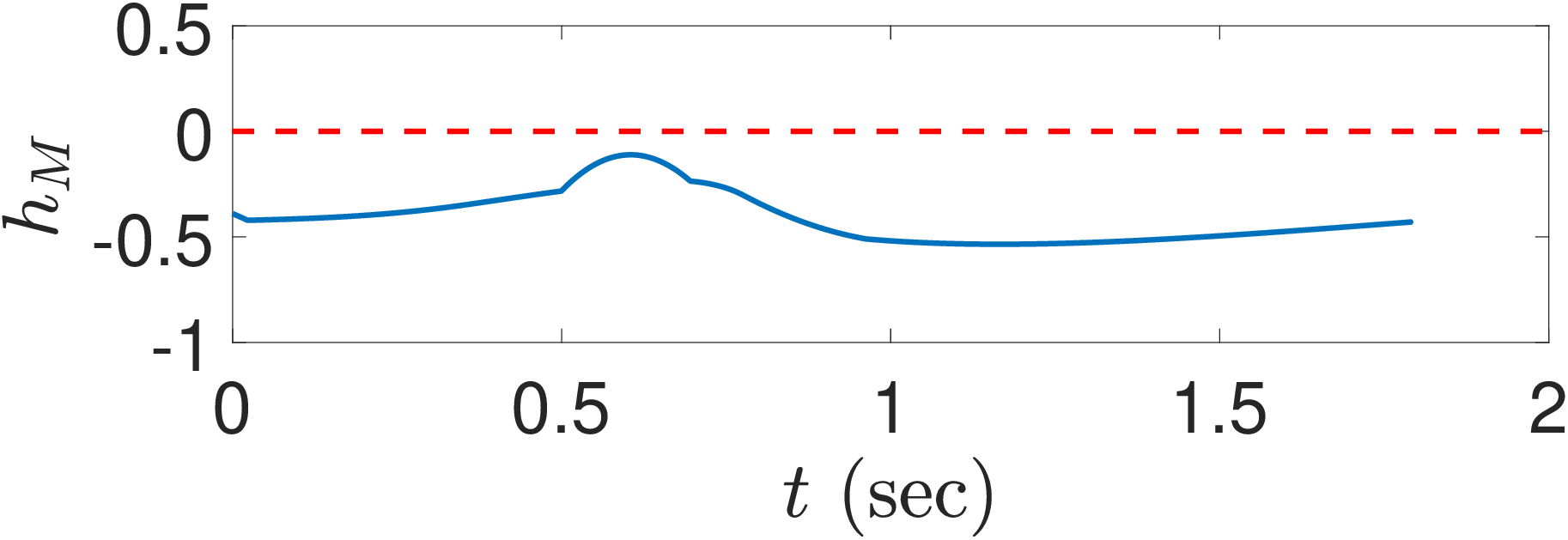}
    \caption{Pointwise maximum $h_M =\max\{h_{ij},h_1,h_2,h_\theta\}$ for the two agents.}\label{fig: max h}
\end{figure}

We choose $u_M = 2, \omega_M = 5, T =2, \mu = 5$, so that $\gamma_1 = 1.2,\gamma_2 = 0.8,\alpha_1 = \alpha_2 = \frac{5\pi}{4}$. The safety distance is chosen as $d_m = 0.1$. Figure \ref{fig: ex3 traj} plots the closed-loop trajectories of the agent and shows that the agent visit the required sets, while remaining inside the safe region and maintaining the safe distance with each other at all times. This is also evident from Figure \ref{fig: max h}, where the pointwise maximum of all the barrier functions (i.e., $h_{ij},h_1, h_2, h_\theta$) is plotted. Since $h_M<0$ at all times, it implies that the both agents satisfy the safety requirements at all times.

\begin{figure}[b]
    \centering
    \includegraphics[ width=1\columnwidth,clip]{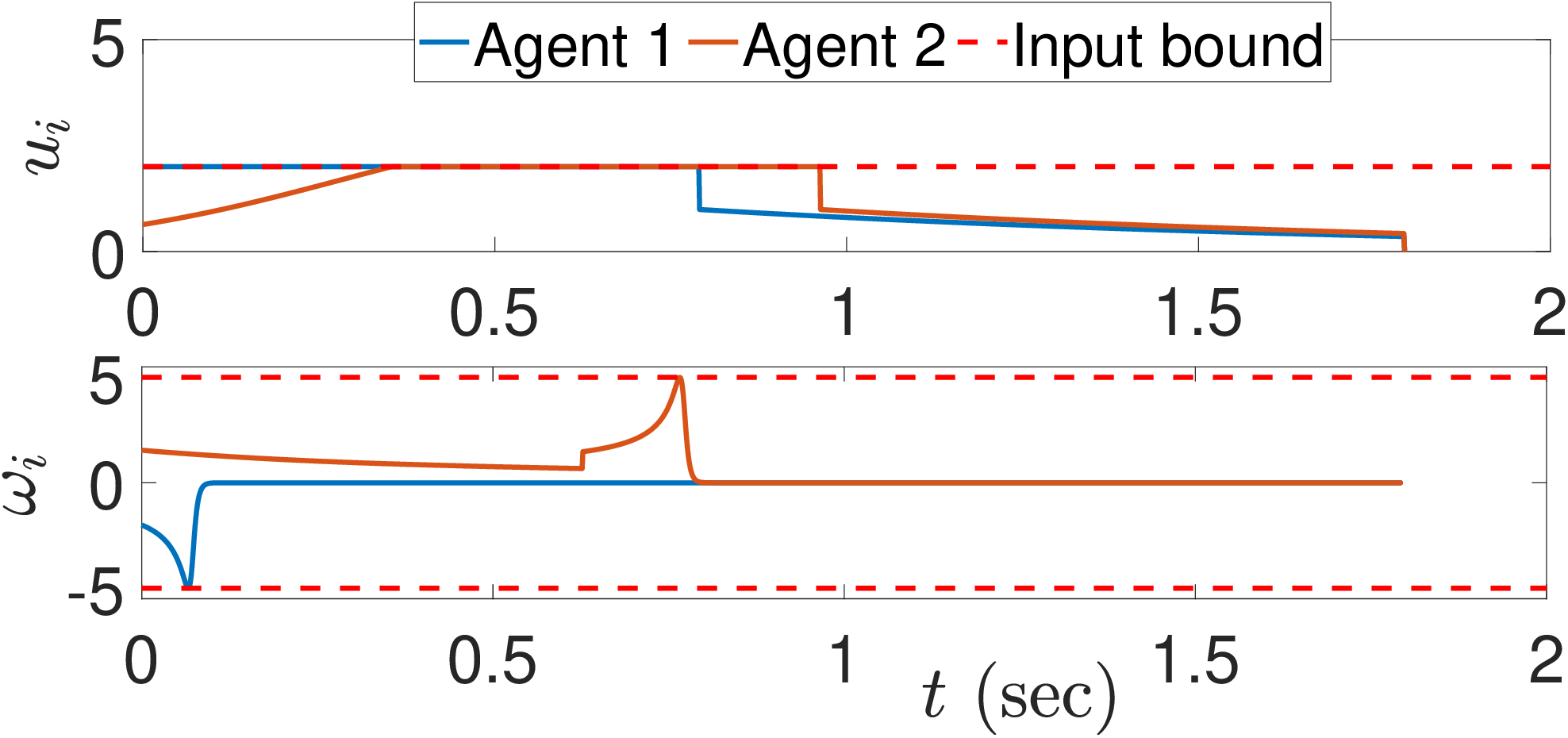}
    \caption{Inputs $u_i$ and $\omega_i$ for the two agents.}\label{fig: two agent input bounds}
\end{figure}

Figure \ref{fig: two agent input bounds} plots the individual inputs of the two agents. It is evident from the figure that the input constraints for the agents are satisfied at all times. Furthermore, note that the linear speeds $u_1, u_2$ go to zero before $t = 2$, implying that the agents reach their respective goal sets within the user-defined time $T = 2$. }

{Finally, we ran the single agent case (only Agent 1 is considered)  for various values of the input bounds on the linear speed $u_M\in [0.02 , 0.1]$ and the required time of convergence $T\in [2, 10]$ for 200 randomly chosen initial conditions $[x_1(0)\; y_1(0)]^T\notin C_2$. Figure \ref{fig: DoA u T} plots the value of $c$ such that the $c-$sublevel set of the Lyapunov function $V_1$ is a FxT-DoA for the system dynamics in \eqref{eq: unicycle dyn}. It can be seen that for a fixed value of $u_M$, the value of $c$ increases as the required time of convergence increases (see the solid lines in the figure). Furthermore, it can be observed that for a given value of the required time of convergence $T$, the value of $c$ increases as the input bound $u_M$ increases (see the dotted vertical lines in the figure). This demonstrates that the domain of attraction for fixed-time stability expands if $T$ or $u_M$ increases, as discussed in Section \ref{sec: DoA T U rel}.}

\begin{figure}[t]
    \centering
    \includegraphics[ width=1\columnwidth,clip]{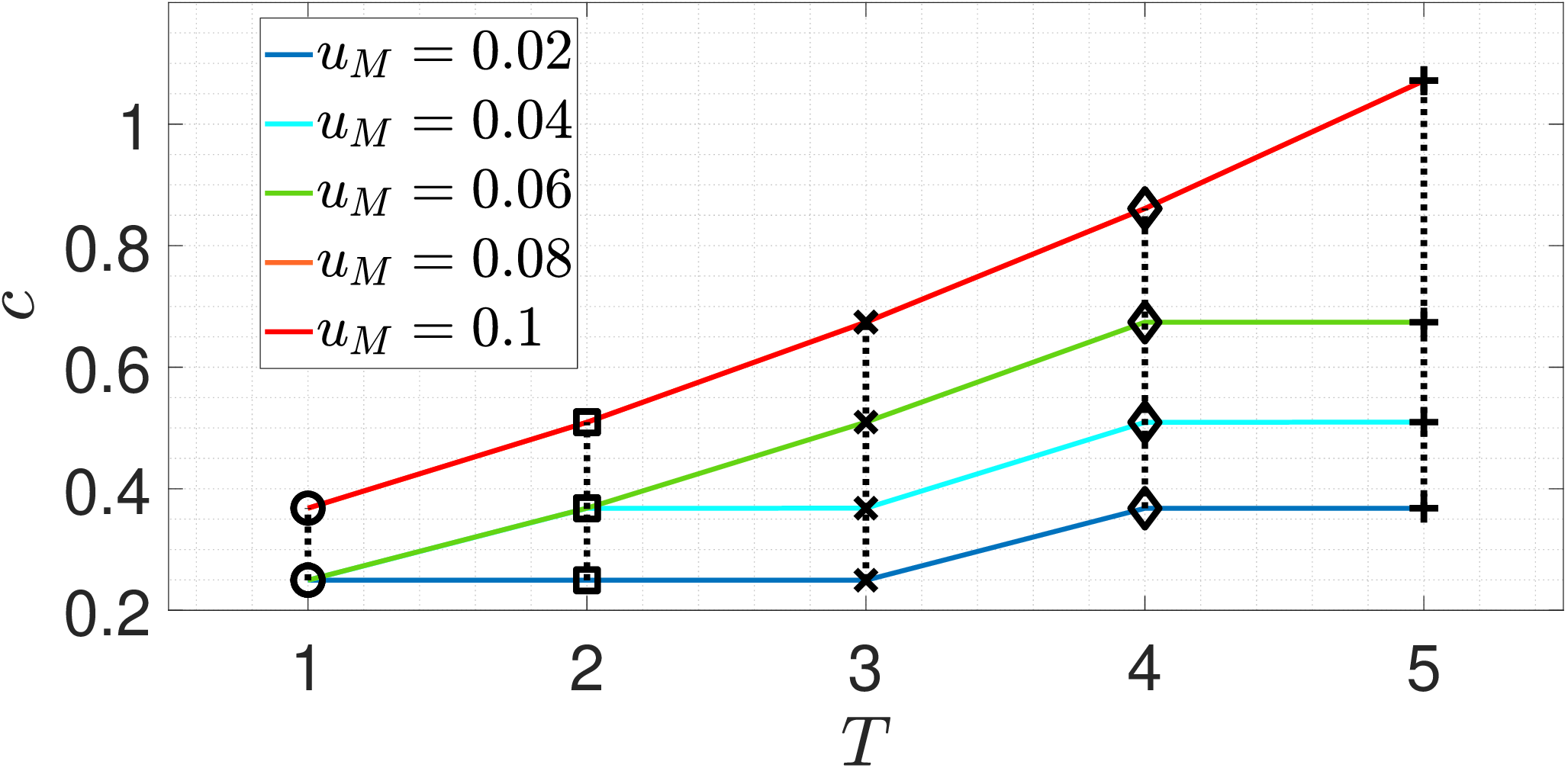}
    \caption{The value of $c$ such that $D = \{x \; | \; V_1(x)\leq c\}$ is the FxT-DoA for various values of $u_M$ and $T$.}\label{fig: DoA u T}
\end{figure}

\section{Conclusions}\label{sec: conclusion}
In this paper, 
we considered the problem of satisfying spatiotemporal constraints requiring that the closed-loop trajectories of a class of nonlinear, control-affine systems remain in a safe set at all times, and reach a goal set within a fixed time in the presence of control input constraints. We established the relation between the domain of attraction for fixed-time stability, the input bounds and the time of convergence, showing that relaxing the time constraint or increasing the input bound results into a larger FxT-DoA. Then, we proposed a novel QP formulation, proved its feasibility under the assumption of existence of a control input that renders the safe set forward invariant, and showed continuity of the solution of the proposed QP.
In the future, we would like to study the spatiotemporal control synthesis for large-scale multi-agent systems with concurrent consideration of switching in the dynamics or the system states. It will be interesting to see how the proposed method extends to systems with non-smooth dynamics, and how to formulate efficient optimization methods under such spatiotemporal constraints. 




\bibliographystyle{plain}  
\bibliography{myreferences}

\appendix
\section{Proof of Lemma \ref{lemma: delta closed form}}\label{app: KKT analysis}
Intuitively, for given control input bounds, a larger value of $T_{ud}$ (which results into smaller values of $\alpha_1, \alpha_2$), i.e., relaxation of time of convergence, should result in satisfaction of \eqref{C2 stab const 1d} with smaller value of $\delta_1$. Conversely, for a given $T_{ud}$ (and thus, for a given pair $\alpha_1, \alpha_2$), a larger control authority should result into satisfaction of \eqref{C2 stab const 1d} with smaller $\delta_1$. In order to verify the intuition , we can compute the closed-form solution of \eqref{QP 1d} for the case when the control input constraint is active, and see how the parameters $T_{ud}, u_m, u_M$ affect the optimal value of $\delta_1$. To this end, consider the Lagrangian of the {QP} in \eqref{QP 1d}:
\begin{align}
\begin{split}
    L \coloneqq &\frac{1}{2}u^2  +\frac{1}{2}\delta_1^2 +c\delta_1 +\lambda_2(u-u_M)  + \lambda_3(u_m-u)\\
    & + \lambda_1( L_fV + L_gV(x)u - \delta_1 V+\alpha_1V^{\gamma_1}  +\alpha_2V^{\gamma_2}).
\end{split}
\end{align}
Now, in order to see the effect of how input constraints affect $\delta_1$, the case when the constraint $u = u_M$ is active is studied under the assumption that $u_M>0$. Lemma \ref{lemma: qp 1d feas} guarantees feasibility of the {QP} in \eqref{QP 1d} for all $x\notin S_G$. Thus, the Slater's condition holds and the {KKT} conditions are both necessary and sufficient for optimality (see e.g., \cite[Chapter 5]{boyd2004convex}). Using the {KKT} conditions, it follows that the optimal solution $(u^\star, \delta_1^\star, \lambda_1^\star, \lambda_2^\star,\lambda_3^\star)$ satisfies
\begin{align*}
    \delta_1^\star(x) & = -c+\lambda_1^\star(x) V(x), \\
    u^\star(x) & = -\lambda_2^\star(x) + \lambda_3^\star(x)-\lambda_1^\star(x) L_gV(x), \\
    \lambda_1^\star(x) & \geq 0, \quad \lambda_2^\star(x)\geq 0, \quad \lambda_3^\star(x) \geq 0,
\end{align*}
for any $x\notin S_G$. We are now ready to present the proof of Lemma \ref{lemma: delta closed form}. 

\begin{pf}
For $u^\star(x) = u_M$, it is required that $\lambda_2^\star(x)>0$. Since $u^\star(x) = u_M$ and $u_m<u_M$, it follows that $u^\star(x)>u_m$ (i.e., the lower-bound constraint is inactive) and so $ \lambda_3^\star(x) = 0$. It follows that $\lambda_2^\star(x) = -u_M-\lambda_1^\star(x) L_gV(x)$. 

Since $u^\star(x) = u_M$, the constraint \eqref{C2 stab const 1d} must be active. Otherwise, we have $\lambda_1^\star(x) = 0$, which implies that $\lambda_2^\star(x) = -u_M<0$, which violates the optimality condition $\lambda_2^\star(x)\geq 0$. Thus, for $\lambda_1^\star(x)>0$ when $u^\star(x) = u_M$, it is essential that the constraint \eqref{C2 stab const 1d} is active, and it follows that the optimal value of $\delta_1$ is given as:{\small\vspace{-15pt}
\begin{align*}
    \delta_1^\star (x)  = & \frac{L_fV(x) + L_gV(x)u_M+\alpha_1V(x)^{\gamma_1}  +\alpha_2V(x)^{\gamma_2}}{V(x)}\nonumber\\
     = & \frac{L_fV(x)+L_gV(x)u_M}{V(x)} +\alpha_1V(x)^{\gamma_1-1} + \alpha_2V(x)^{\gamma_2-1}.
\end{align*}}\normalsize
Using this, and the definition of function $a(x)$, it follows that $$\lambda_1^\star(x) = \frac{a(x)}{V(x)^2}, \quad \lambda_2^\star(x) = -u_m-L_gV(x)\frac{a(x)}{V(x)^2}.$$ 
Now, for $x\in S_M$, it holds that $\lambda_1^\star(x)> 0$ and $\lambda_2^\star(x)>0$ and thus, the optimal value of $\delta_1$ is given by \eqref{eq: delta closed 1d} and that the optimal value of $u$ being $u^\star = u_M$ holds when $x\in S_M$.\footnote{If the set $S_M = \emptyset$, it implies that there does not exist $x$ such that $u^\star(x) = u_M$, or in other words, the control input never saturates with the upper bound.} 
\end{pf}

\section{Proof of Theorem \ref{Th: QP sol cont}}\label{app: proof Thm QP sol cont}

\begin{pf}
The proof is based on \cite[Theorem 2.1]{robinson1974perturbed}. Denote by $I(x)$, the indices of rows of matrix $A(x)$ corresponding to the active constraints, i.e., $j\in I(x)$ implies $A_{j}(x)z^\star (x) = b_{j}(x)$, where $A_j\in \mathbb R^{1\times m+2}$ is the $j$-th row of the matrix $A$ and $b_j\in \mathbb R$ the $j$-th element of $b$. Define matrix $A_{ac}$ and $b_{ac}$ by collecting $A_j(x)$, and of $b_j$, respectively, so that $A_{ac}(x)z^\star (x) = b_{ac}(x)$. Since at most one of the input constraints $u_i\leq u_{M_i}$ or $u_{m_i}\leq u_i$ can be active at any given time, the matrix $A_{ac}(x)$ has $k$ rows from $\begin{bmatrix}A_u & \mathbf 0_{2m} & \mathbf 0_{2m}\end{bmatrix}$, where $k\leq m$, which are linearly independent. Furthermore, it has $p$ rows from $\begin{bmatrix}L_gh_G & -h_G & 0\\  L_gh_S & 0 & h_S\end{bmatrix}$, where $p\leq 2$. Since $h_G, h_S\neq 0$ for $x\in \textnormal{int}(S_S)\setminus S_G$, these $k+p$ rows are linearly independent. Thus, the matrix $A_{ac}$ is full row-rank, i.e., the gradients of the active constraints $\{A_{ac_i}(x)\}$, where $A_{ac_i}(x)$ is the $i-$th row of matrix $A_{ac}(x)$, are linearly independent. 

The second derivative $\nabla_{zz}L$ of the Lagrangian defined as $$L(z,x,\lambda) \coloneqq \frac{1}{2}z^THz+F^Tz + \lambda^T(A(x)z-b(x)),$$ with respect to $z$ is $H$, which is a positive definite matrix. Using this, and the fact that the {QP} \eqref{QP gen} is feasible, it holds that the second-order sufficient conditions for optimality hold (see e.g. \cite[Section 2.3]{robinson1974perturbed}). Note that \cite[Theorem 2.1]{robinson1974perturbed} requires that the objective function and the functions $G_i(x,u)$ have the second derivatives jointly continuous in $(x,u)$. Since the objective function $\frac{1}{2}z^THz+F^Tz$ is independent of $x$, and the constraint functions $G_i(x,u)$ are linear in $u$, the second derivative of these functions are independent of $x$, and thus, satisfy this condition trivially. Finally, the strict complementary slackness condition is satisfied per Assumption \ref{assum hs hg scs}. Thus, all the conditions of \cite[Theorem 2.1]{robinson1974perturbed} are satisfied. Therefore, for every $x\in \textnormal{int}(S_S)\setminus S_G$, there exists an open neighborhood $\mathcal X\subset \textnormal{int}(S_S)\setminus S_G$ of $x$ such that the solution $z^\star (x)$ is continuous for all $x\in \mathcal X$. Since this holds for all $x\in \textnormal{int}(S_S)\setminus S_G$, it follows that the solution $z^\star (x)$ is continuous for all $x\in \textnormal{int}(S_S)\setminus S_G$.
\end{pf}

\end{document}